\documentclass{article}

\usepackage{amsmath,amsfonts,bm}

\def\eqref#1{equation~\ref{#1}}
\def\1{\bm{1}}

\DeclareMathAlphabet{\mathsfit}{\encodingdefault}{\sfdefault}{m}{sl}
\SetMathAlphabet{\mathsfit}{bold}{\encodingdefault}{\sfdefault}{bx}{n}

\usepackage[accepted]{icml2025}

\usepackage{microtype}
\usepackage{graphicx}
\usepackage{booktabs} % for professional tables
\usepackage{bbding}
\usepackage{hyperref}

\usepackage{comment}
\usepackage{hyperref}
\usepackage{url}
\usepackage{bbm}
\usepackage{amsthm}
\usepackage{makecell}
\usepackage{graphicx}

\usepackage{algorithm}

\usepackage{mathtools}
\usepackage{unicode}
\usepackage{booktabs}
\usepackage{multirow}
\usepackage{csquotes}
\usepackage{subcaption}
\usepackage{rotating}

\usepackage[capitalize,noabbrev]{cleveref}

\theoremstyle{plain}
\newtheorem{theorem}{Theorem}[section]

\newtheorem{lemma}[theorem]{Lemma}

\theoremstyle{definition}
\newtheorem{definition}[theorem]{Definition}

\theoremstyle{remark}

\usepackage[textsize=tiny]{todonotes}

\icmltitlerunning{ROS: A GNN-based Relax-Optimize-and-Sample Framework for Max-$k$-Cut Problems}

\begin{document}

\twocolumn[
\icmltitle{ROS: A GNN-based Relax-Optimize-and-Sample Framework for Max-$k$-Cut Problems}

% It is OKAY to include author information, even for blind
% submissions: the style file will automatically remove it for you
% unless you've provided the [accepted] option to the icml2025
% package.

% List of affiliations: The first argument should be a (short)
% identifier you will use later to specify author affiliations
% Academic affiliations should list Department, University, City, Region, Country
% Industry affiliations should list Company, City, Region, Country

% You can specify symbols, otherwise they are numbered in order.
% Ideally, you should not use this facility. Affiliations will be numbered
% in order of appearance and this is the preferred way.
% \icmlsetsymbol{equal}{*}

\begin{icmlauthorlist}
\icmlauthor{Yeqing Qiu}{sribd,cuhksz}
\icmlauthor{Ye Xue}{sribd,cuhksz}
\icmlauthor{Akang Wang}{sribd,cuhksz}
\icmlauthor{Yiheng Wang}{sribd,cuhksz}
\icmlauthor{Qingjiang Shi}{sribd,tju}
\icmlauthor{Zhi-Quan Luo}{sribd,cuhksz}
%\icmlauthor{}{sch}
%\icmlauthor{}{sch}
\end{icmlauthorlist}

\icmlaffiliation{cuhksz}{The Chinese University of Hong Kong, Shenzhen, China.}
\icmlaffiliation{sribd}{Shenzhen Research Institute of Big Data, Shenzhen, China.}
\icmlaffiliation{tju}{Tongji University, Shanghai, China}

\icmlcorrespondingauthor{Ye Xue}{xueye@cuhk.edu.cn}

% You may provide any keywords that you
% find helpful for describing your paper; these are used to populate
% the "keywords" metadata in the PDF but will not be shown in the document
\icmlkeywords{Machine Learning, ICML}

\vskip 0.3in
]

% this must go after the closing bracket ] following \twocolumn[ ...

% This command actually creates the footnote in the first column
% listing the affiliations and the copyright notice.
% The command takes one argument, which is text to display at the start of the footnote.
% The \icmlEqualContribution command is standard text for equal contribution.
% Remove it (just {}) if you do not need this facility.

\printAffiliationsAndNotice{}  % leave blank if no need to mention equal contribution
% \printAffiliationsAndNotice{\icmlEqualContribution} % otherwise use the standard text.

\begin{abstract}
The Max-$k$-Cut problem is a fundamental combinatorial optimization challenge that generalizes the classic $\mathcal{NP}$-complete Max-Cut problem. While relaxation techniques are commonly employed to tackle Max-$k$-Cut, they often lack guarantees of equivalence between the solutions of the original problem and its relaxation. To address this issue, we introduce the Relax-Optimize-and-Sample (ROS) framework. In particular, we begin by relaxing the discrete constraints to the continuous probability simplex form. Next, we pre-train and fine-tune a graph neural network model to efficiently optimize the relaxed problem. Subsequently, we propose a sampling-based construction algorithm to map the continuous solution back to a high-quality Max-$k$-Cut solution. By integrating geometric landscape analysis with statistical theory, we establish the consistency of function values between the continuous solution and its mapped counterpart. Extensive experimental results on random regular graphs, the Gset benchmark, and the real-world datasets demonstrate that the proposed ROS framework effectively scales to large instances with up to $20,000$ nodes in just a few seconds, outperforming state-of-the-art algorithms. Furthermore, ROS exhibits strong generalization capabilities across both in-distribution and out-of-distribution instances, underscoring its effectiveness for large-scale optimization tasks.
\end{abstract}

\section{Introduction}
\label{sec:intro}

% max-k-cut, definition
% lots of application
% NP-complete
% Max-Cut as a special case

The \textit{Max-$k$-Cut problem} involves partitioning the vertices of a graph into $k$ disjoint subsets in such a way that the total weight of edges between vertices in different subsets is maximized. 
This problem represents a significant challenge in combinatorial optimization and finds applications across various fields, including telecommunication networks~\cite{eisenblatter2002semidefinite, gui2018pci2}, data clustering~\cite{poland2006clustering, ly2023data}, and theoretical physics~\cite{cook2019gpu, coja2022ising}. 
The Max-$k$-Cut problem is known to be $\mathcal{NP}$-complete, as it generalizes the well-known \textit{Max-Cut problem}, which is one of the 21 classic $\mathcal{NP}$-complete problems identified by~\citet{karp2010reducibility}.

% Classic algorithms on Max-k-cut / Max-Cut
% exact algorithm
% approximation algorithm
% heuristic algorithm

Significant efforts have been made to develop methods for solving Max-$k$-Cut problems~\cite{nath2024benchmark}. 
\citet{ghaddar2011branch} introduced an exact branch-and-cut algorithm based on semi-definite programming, capable of handling graphs with up to $100$ vertices. 
For larger instances, various polynomial-time approximation algorithms have been proposed. 
\citet{goemans1995improved} addressed the Max-Cut problem by first solving a semi-definite relaxation to obtain a fractional solution, then applying a randomization technique to convert it into a feasible solution, resulting in a $0.878$-approximation algorithm. 
Building on this, \citet{frieze1997improved} extended the approach to Max-$k$-Cut, offering feasible solutions with approximation guarantees. 
\citet{de2004approximate} further improved these guarantees, while \citet{shinde2021memory} optimized memory usage. Despite their strong theoretical performance, these approximation algorithms involve solving computationally intensive semi-definite programs, rendering them impractical for large-scale Max-$k$-Cut problems.
A variety of heuristic methods have been developed to tackle the scalability challenge. 
For the Max-Cut problem, \citet{burer2002rank} proposed rank-two relaxation-based heuristics, and \citet{goudet2024large} introduced a meta-heuristic approach using evolutionary algorithms. 
For Max-$k$-Cut, heuristics such as genetic algorithms~\cite{panxing2016pci}, greedy search~\cite{gui2018pci2}, multiple operator heuristics~\cite{ma2017multiple}, and local search~\cite{garvardt2023parameterized} have been proposed. While these heuristics can handle much larger Max-$k$-Cut instances, they often struggle to balance efficiency and solution quality.

Recently, \textit{machine learning} techniques have gained attention for enhancing optimization algorithms~\cite{bengio2021machine, gasse2022machine, chen2024learning}. 
Several studies, including~\citet{khalil2017learning, barrett2020exploratory, chen2020heuristic, barrett2022learning}, framed the Max-Cut problem as a sequential decision-making process, using reinforcement learning to train policy networks for generating feasible solutions. 
However, RL-based methods often suffer from extensive sampling efforts and increased complexity in action space when extended to Max-$k$-Cut, and hence entails significantly longer training and testing time. 
\citet{karalias2020erdos} focuses on subset selection, including Max-Cut as a special case. It trains a \textit{graph neural network}~(GNN) to produce a distribution over subsets of nodes of an input graph by minimizing a probabilistic penalty loss function.
After the network has been trained, a randomized algorithm is employed to sequentially decode a valid Max-Cut solution from the learned distribution.
A notable advancement by~\citet{schuetz2022combinatorial} reformulated Max-Cut as a quadratic unconstrained binary optimization (QUBO), removing binarity constraints to create a differentiable loss function. 
This loss function was used to train a GNN, followed by a simple projection onto integer variables after unsupervised training. The key feature of this approach is solving the Max-Cut problem during the training phase, eliminating the need for a separate testing stage.
Although this method can produce high-quality solutions for Max-Cut instances with millions of nodes, the computational time remains significant due to the need to optimize a parameterized GNN from scratch.
The work of~\citet{tonshoff2022one} first formulated the Max-Cut problem as a \textit{Constraint Satisfaction Problem} (CSP) and then proposed a novel GNN-based reinforcement learning approach. This method outperforms prior neural combinatorial optimization techniques and conventional search heuristics. However, to the best of our knowledge, it is limited to unweighted Max-$k$-Cut problems. NeuroCUT~\cite{shah2024neurocut} is a partitioning method based on reinforcement learning, whereas DGCLUSTER~\cite{bhowmick2024dgcluster} and DMoN~\cite{tsitsulin2023graph} utilize GNNs to optimize clustering objectives. However, these methods are specifically designed for graph clustering, which focuses on minimizing inter-cluster connections—contrary to Max-$k$-Cut, where the goal is to maximize inter-partition connections. Consequently, they are not directly applicable to our problem. Although NeuroCUT claims support for arbitrary objective functions, its node selection heuristics are tailored exclusively for graph clustering, rendering it unsuitable for Max-$k$-Cut.

\begin{figure*}[t]
  \centering
  \includegraphics[width=\linewidth]{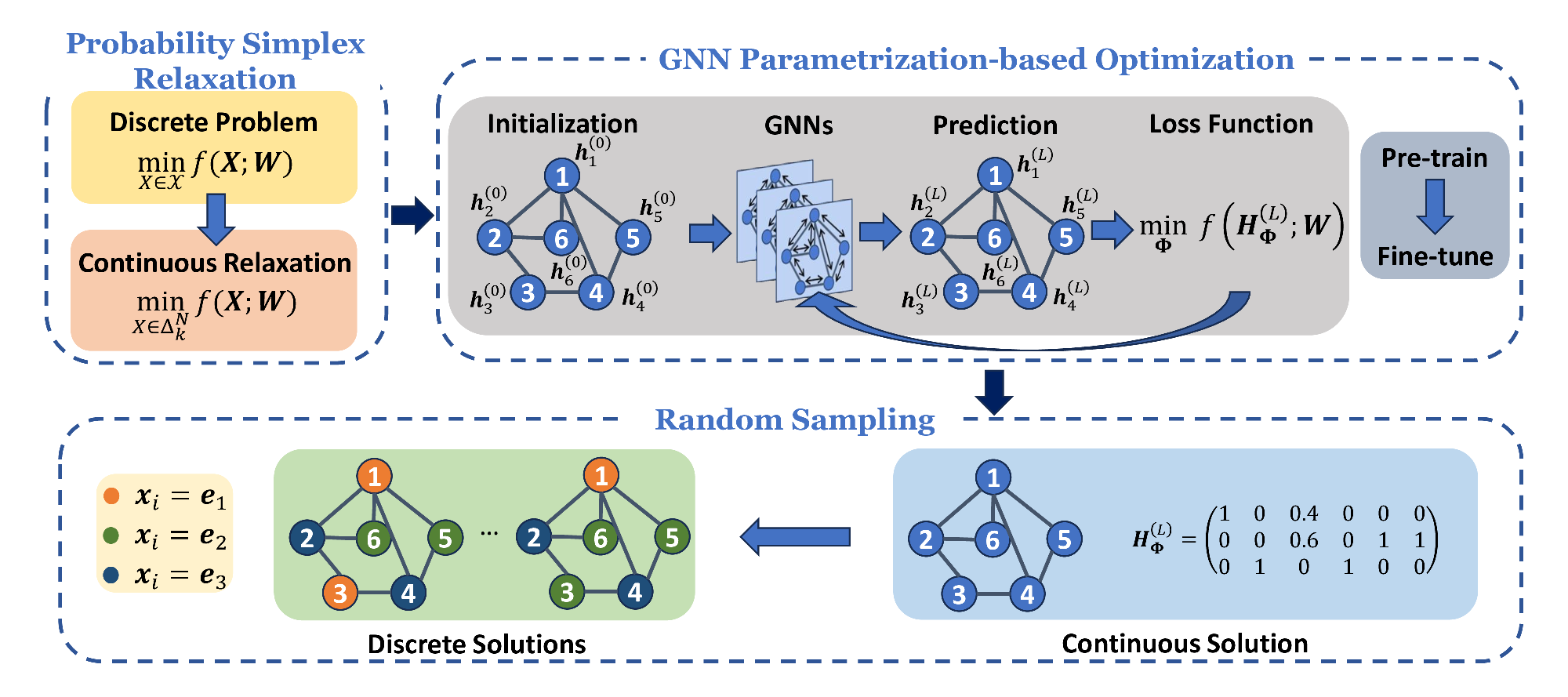}
  \caption{The Relax-Optimize-and-Sample framework.}
  \label{fig:relax-and-optimize_framework}
\end{figure*}

In this work, we propose a GNN-based \textit{Relax-Optimize-and-Sample}~(ROS) framework for efficiently solving the Max-$k$-Cut problem with arbitrary edge weights. The framework is depicted in Figure~\ref{fig:relax-and-optimize_framework}. Initially, the Max-$k$-Cut problem is formulated as a discrete optimization task. To handle this, we introduce \textit{probability simplex relaxations}, transforming the discrete problem into a continuous one. We then optimize the relaxed formulation by training parameterized GNNs in an unsupervised manner. To further improve efficiency, we apply \textit{transfer learning}, utilizing pre-trained GNNs to warm-start the training process. Finally, we refine the continuous solution using a \textit{random sampling algorithm}, resulting in high-quality Max-$k$-Cut solutions.

The key contributions of our work are summarized as follows:
\begin{itemize}
    \item \textbf{Novel Framework.} We propose a scalable ROS framework tailored to the weighted Max-$k$-Cut problem with arbitrary signs, built on solving continuous relaxations using efficient learning-based techniques.
    
    \item \textbf{Theoretical Foundations.} We conduct a rigorous theoretical analysis of both the relaxation and sampling steps. By integrating geometric landscape analysis with statistical theory, we demonstrate the consistency of function values between the continuous solution and its sampled discrete counterpart.
    
    \item \textbf{Superior Performance.} Comprehensive experiments on public benchmark datasets show that our framework produces high-quality solutions for Max-$k$-Cut instances with up to $20,000$ nodes in just a few seconds. Our approach significantly outperforms state-of-the-art algorithms, while also demonstrating strong generalization across various instance types.
\end{itemize}

\section{Preliminaries\label{sec:prel}}

\subsection{Max-\texorpdfstring{$k$}{k}-Cut Problems}
\label{section:max-k-cut_definition}

Let $\mathcal{G} = (\mathcal{V}, \mathcal{E})$ represent an undirected graph with vertex set $\mathcal{V}$ and edge set $\mathcal{E}$. 
%Each edge $(i, j) \in \mathcal{E}$ is assigned a non-negative weight $\bm{W}_{ij}$. 
Each edge $(i, j) \in \mathcal{E}$ is assigned an arbitrary weight $\bm{W}_{ij} \in \mathbb{R}$, which can have any sign. 
A \textit{cut} in $\mathcal{G}$ refers to a partition of its vertex set. The Max-$k$-Cut problem involves finding a $k$-partition $(\mathcal{V}_1, \dots, \mathcal{V}_k)$ of the vertex set $\mathcal{V}$ such that the sum of the weights of the edges between different partitions is maximized.

To represent this partitioning, we employ a $k$-dimensional one-hot encoding scheme. Specifically, we define a $k \times N$ matrix $\bm{X} \in \mathbb{R}^{k \times N}$ where each column represents a one-hot vector. The Max-$k$-Cut problem can be formulated as:
\begin{equation}
    \begin{aligned}
        & \underset{\bm{X} \in \mathbb{R}^{k \times N} }{\text{max}}  && \frac{1}{2} \sum_{i=1}^N \sum_{j=1}^N \bm{W}_{ij} \left( 1 - \bm{X}_{\cdot i}^\top \bm{X}_{\cdot j} \right) \\
        & \quad \text{s. t.}  && \bm{X}_{\cdot j} \in \{ \bm{e}_1, \bm{e}_2, \dots, \bm{e}_k \} \quad\quad \forall j \in \mathcal{V},
    \end{aligned}  \label{eq:max-k-cut_formulation}
\end{equation}
where $\bm{X}_{\cdot j}$ denotes the $j^{th}$ column of $\bm{X}$, $\bm{W}$ is a symmetric matrix with zero diagonal entries, and $\bm{e}_\ell \in \mathbb{R}^k$ is a one-hot vector with the $\ell^{th}$ entry set to $1$. This formulation aims to maximize the total weight of edges between different partitions, ensuring that each node is assigned to exactly one partition, represented by the one-hot encoded vectors.
We remark that weighted Max-$k$-Cut problems with arbitrary signs is a generalization of classic Max-Cut problems and arise in many interesting applications~\cite{de1995exact, poland2006clustering,hojny2021mixed}.

\subsection{Graph Neural Networks}
\label{section:GNN}

GNNs are powerful tools for learning representations from graph-structured data. 
GNNs operate by iteratively aggregating information from a node's neighbors, enabling each node to capture increasingly larger sub-graph structures as more layers are stacked. 
This process allows GNNs to learn complex patterns and relationships between nodes, based on their local connectivity.

At the initial layer ($l = 0$), each node $i \in \mathcal{V}$ is assigned a feature vector $\bm{h}_i^{(0)}$, which typically originates from node features or labels. 
The representation of node $i$ is then recursively updated at each subsequent layer through a parametric aggregation function $f_{\bm{\Phi}^{(l)}}$, defined as:
\begin{align}
    \bm{h}_i^{(l)} = f_{\bm{\Phi}^{(l)}} \left( \bm{h}_i^{(l-1)}, \{\bm{h}_j^{(l-1)} : j \in \mathcal{N}(i)\} \right),
\end{align}
where $\bm{\Phi}^{(l)}$ represents the trainable parameters at layer $l$, $\mathcal{N}(i)$ denotes the set of neighbors of node $i$, and $\bm{h}_i^{(l)}$ is the node's embedding at layer $l$ for $l \in \{1, 2, \cdots, L\}$. 
This iterative process enables the GNN to propagate information throughout the graph, capturing both local and global structural properties.

\section{A Relax-Optimize-and-Sample Framework}

In this work, we leverage continuous optimization techniques to tackle Max-$k$-Cut problems, introducing a novel ROS framework. 
Acknowledging the inherent challenges of discrete optimization, we begin by relaxing the problem to probability simplices and concentrate on optimizing this relaxed version. 
To achieve this, we propose a machine learning-based approach. 
Specifically, we model the relaxed problem using GNNs, pre-training the GNN on a curated graph dataset before fine-tuning it on the specific target instance. 
After obtaining high-quality solutions to the relaxed continuous problem, we employ a random sampling procedure to derive a discrete solution that preserves the same objective value.

\subsection{Probability Simplex Relaxations}
\label{sec:probability_simplex_relaxation}

To simplify the formulation of the problem~(\ref{eq:max-k-cut_formulation}), we remove constant terms and negate the objective function, yielding an equivalent formulation expressed as follows:
\begin{equation*}
    \begin{aligned}
      & \underset{\bm{X} \in \mathcal{X}}{\text{min}} \quad f(\bm{X};\bm{W}) \coloneqq \text{Tr}(\bm{X}\bm{W}\bm{X}^\top), \hspace{2cm} (\textbf{P})
    \end{aligned}
\end{equation*}
where $\mathcal{X} \coloneqq \left\{\bm{X} \in \mathbb{R}^{k \times N}: \bm{X}_{\cdot j} \in \{\bm{e}_1, \bm{e}_2, \dots, \bm{e}_k\}, \forall j \in \mathcal{V} \right\}$. It is important to note that the matrix $\bm{W}$ is indefinite due to its diagonal entries being set to zero.

Given the challenges associated with solving the discrete problem \textbf{P}, we adopt a naive relaxation approach, obtaining the convex hull of $\mathcal{X}$ as the Cartesian product of $N$ $k$-dimensional probability simplices, denoted by $\Delta_k^N$. Consequently, the discrete problem \textbf{P} is relaxed into the following continuous optimization form:
\begin{equation*}
    \begin{aligned}
         \underset{\bm{X} \in \Delta_k^N }{\text{min}} & \quad  f(\bm{X}; \bm{W}). \hspace{4.5cm} (\overline{\textbf{P}})
    \end{aligned}
\end{equation*}
Before optimizing problem~$\overline{\textbf{P}}$, we will characterize its \textit{geometric landscape}. To facilitate this, we introduce the following definition.

\begin{definition} \label{definition:neighborhood}
    Let $\overline{\bm{X}}$ denote a point in $\Delta_k^N$. We define the neighborhood induced by $\overline{\bm{X}}$ as follows: 
    \begin{equation*}
        \begin{aligned}
            \mathcal{N}(\overline{\bm{X}}) \coloneqq \left\{\bm{X} \in \Delta_k^N  
                \left|  
                \begin{aligned}
                     &  \sum_{i \in \mathcal{K}(\overline{\bm{X}}_{\cdot j})} \bm{X}_{ij} = 1, && \forall j \in \mathcal{V}  
                \end{aligned}
                \right.
            \right\},
        \end{aligned}
    \end{equation*}    
    where $\mathcal{K}(\overline{\bm{X}}_{\cdot j}) \coloneqq \{i\in \{1, \ldots, k\} \mid \overline{\bm{X}}_{ij} > 0  \}$.
\end{definition}

The set $\mathcal{N}(\overline{\bm{X}})$ represents a neighborhood around $\overline{\bm{X}}$, where each point in $\mathcal{N}(\overline{\bm{X}})$ can be derived by allowing each non-zero entry of the matrix $\overline{\bm{X}}$ to vary freely, while the other entries are set to zero. Utilizing this definition, we can establish the following theorem.

\begin{theorem}
\label{theorem:local_optimal_neighborhood}
    Let $\bm{X}^\star$ denote a globally optimal solution to $\overline{\textbf{P}}$, and let $\mathcal{N}(\bm{X}^\star)$ be its induced neighborhood. Then 
    \begin{equation*}
        \begin{aligned}
           & f(\bm{X}; \bm{W}) = f(\bm{X}^\star; \bm{W}), &&  \forall \bm{X} \in \mathcal{N}(\bm{X}^{\star}).
        \end{aligned}
    \end{equation*}
\end{theorem}

Theorem~\ref{theorem:local_optimal_neighborhood} states that for a globally optimal solution $\bm{X}^{\star}$, every point within its neighborhood $\mathcal{N}(\bm{X}^{\star})$ shares the same objective value as $\bm{X}^{\star}$, thus forming a \textit{basin} in the geometric landscape of $f(\bm{X}; \bm{W})$. If $\bm{X}^{\star} \in \mathcal{X}$ (i.e., an integer solution), then $\mathcal{N}(\bm{X}^{\star})$ reduces to the singleton set $\left\{\bm{X}^{\star}\right\}$. Conversely, if $\bm{X}^{\star} \notin \mathcal{X}$, there exist $\prod_{j \in \mathcal{V}}|\mathcal{K}(\bm{X}^{\star}_{\cdot j})|$ unique integer solutions within $\mathcal{N}(\bm{X}^{\star})$ that maintain the same objective value as $\bm{X}^{\star}$. This indicates that once a globally optimal solution to the relaxed problem $\overline{\textbf{P}}$ is identified, it becomes straightforward to construct an optimal solution for the original problem \textbf{P} that preserves the same objective value.

According to~\citet{carlson1966scheduling}, among all globally optimal solutions to the relaxed problem $\overline{\textbf{P}}$, the integer solution always exists. Theorem~\ref{theorem:local_optimal_neighborhood} extends this result, indicating that if the globally optimal solution is fractional, we can provide a straightforward method to derive its integer counterpart.
We remark that it is highly non-trivial to guarantee that the feasible Max-$k$-Cut solution obtained from the relaxation one has the same quality.

\textbf{Example}. 
Consider a Max-Cut problem ($k=2$) associated with the weight matrix $\bm{W}$. We optimize its relaxation and obtain the optimal solution $\bm{X}^\star$.
	\begin{equation*}
		\begin{aligned}
			\bm{W} \coloneqq 
			\begin{pmatrix} 
				0&1&1 \\  
				1&0&1 \\
				1&1&0  
			\end{pmatrix}, 
			\bm{X}^\star \coloneqq 
			     \begin{pmatrix}
			     	  p & 1 & 0 \\ 
			     	  1-p & 0 & 1 
			     \end{pmatrix},
		\end{aligned}
	\end{equation*}
where $p \in \left[0, 1\right]$. From the neighborhood $\mathcal{N}(\bm{X}^{\star})$, we can identify the following integer solutions that maintain the same objective value.
	\begin{equation*}
	\begin{aligned}
		\bm{X}_1^\star = 
		\begin{pmatrix}
			0 & 1 & 0 \\ 
			1 & 0 & 1 
		\end{pmatrix}, 
		\bm{X}_2^\star = 
		\begin{pmatrix}
			1 & 1 & 0 \\ 
			0 & 0 & 1 
		\end{pmatrix}.
	\end{aligned}
\end{equation*}

Given that $\overline{\textbf{P}}$ is a non-convex program, identifying its global minimum is challenging.
Consequently, the following two critical questions arise.
\begin{itemize}
    \item[\textbf{Q1}.] Since solving $\overline{\textbf{P}}$ to global optimality is $\mathcal{NP}$-hard, how to efficiently optimize $\overline{\textbf{P}}$ for high-quality solutions?

    \item[\textbf{Q2}.] Given $\overline{\bm{X}} \in \Delta_k^N \setminus \mathcal{X} $ as a high-quality solution to $\overline{\textbf{P}}$, can we construct a feasible solution $\hat{\bm{X}} \in \mathcal{X}$ to \textbf{P} such that $f(\hat{\bm{X}}; \bm{W}) = f(\overline{\bm{X}}; \bm{W})$? 
\end{itemize}

We provide a positive answer to \textbf{Q2} in Section~\ref{sec:random_sampling}, while our approach to addressing \textbf{Q1} is deferred to Section~\ref{sec:GNN_optimization}.

\subsection{Random Sampling}
\label{sec:random_sampling}

Let $\overline{\bm{X}} \in \Delta_k^N \setminus \mathcal{X}$ be a feasible solution to the relaxation $\overline{\textbf{P}}$. 
Our goal is to construct a feasible solution $\bm{X} \in \mathcal{X}$ for the original problem \textbf{P}, ensuring that the corresponding objective values are equal. Inspired by Theorem~\ref{theorem:local_optimal_neighborhood}, we propose a \textit{random sampling} procedure, outlined in Algorithm~\ref{alg:random}. 
In this approach, we sample each column $\bm{X}_{\cdot i}$ of the matrix $\bm{X}$ from a categorical distribution characterized by the event probabilities $\overline{\bm{X}}_{\cdot i}$ (denoted as $\text{Cat}(\bm{x}; \bm{p} = \overline{\bm{X}}_{\cdot i})$ in Step 3 of Algorithm~\ref{alg:random}). 
This randomized approach yields a feasible solution $\hat{\bm{X}}$ for \textbf{P}. 
However, since Algorithm~\ref{alg:random} incorporates randomness in generating $\hat{\bm{X}}$ from $\overline{\bm{X}}$, the value of $f(\hat{\bm{X}}; \bm{W})$ becomes random as well. 
This raises the critical question: is this value greater or lesser than $f(\overline{\bm{X}}; \bm{W})$? 
We address this question in Theorem~\ref{theorem:equal_obj_value}.

\begin{algorithm}[h]
\caption{Random Sampling}
\label{alg:random}
    \begin{algorithmic}[1]
    \STATE \textbf{Input:} $\overline{\bm{X}} \in \Delta_k^N$  % \Comment{any feasible solution to $\overline{\bm{P}}$ }
        \FOR{$i = 1$ to $N$}      % \Comment{each dimension is independent}      
        
            \STATE {$\hat{\bm{X}}_{\cdot i} \sim \text{Cat}(\bm{x}; \bm{p} = \overline{\bm{X}}_{\cdot i})$ } \label{step:sample} 
            % \Comment{sampling from a categorical distribution}
            % 
        \ENDFOR
    \STATE{\textbf{Output:} $\hat{\bm{X}} \in \mathcal{X}$ }   
    
    \end{algorithmic}
\end{algorithm}

\begin{theorem}
    Let $\overline{\bm{X}}$ and $\hat{\bm{X}}$ denote the input and output of Algorithm~\ref{alg:random}, respectively. Then, we have $\mathbb{E}_{\hat{\bm{X}}}[f(\hat{\bm{X}};\bm{W})]=f(\overline{\bm{X}};\bm{W})$. 
    \label{theorem:equal_obj_value}
\end{theorem}

Theorem~\ref{theorem:equal_obj_value} states that $f(\hat{\bm{X}};\bm{W})$ is equal to $f(\overline{\bm{X}};\bm{W})$ in expectation.
This implies that the random sampling procedure operates on a fractional solution, yielding Max-$k$-Cut feasible solutions with the same objective values in a probabilistic sense. While the Lovász-extension-based method~\cite{bach2013learning} also offers a framework for continuous relaxation, achieving similar theoretical results for arbitrary \(k\) and edge weights \(\bm{W}_{i,j} \in \mathbb{R}\) is not always guaranteed. 
In practice, we execute Algorithm~\ref{alg:random} $T$ times and select the solution with the lowest objective value of $f$ as our best result. 
We remark that the theoretical interpretation in Theorem~\ref{theorem:equal_obj_value} distinguishes our sampling algorithm from the existing ones in the literature~\cite{karalias2020erdos, toenshoff2021graph, michael2024continuous}.

\subsection{GNN Parametrization-Based Optimization}
\label{sec:GNN_optimization}

To solve the problem $\overline{\textbf{P}}$, we propose an efficient learning-to-optimize~(L2O) method based on GNN parametrization. This approach reduces the laborious iterations typically required by classical optimization methods (e.g., mirror descent). Additionally, we introduce a \enquote{pre-train + fine-tune} strategy, where the model is endowed with prior graph knowledge during the pre-training phase, significantly decreasing the computational time required to optimize $\overline{\textbf{P}}$.

\textbf{GNN Parametrization.} 
The Max-$k$-Cut problem can be framed as a node classification task, allowing us to leverage GNNs to aggregate node features, and obtain high-quality solutions. Initially, we assign a random embedding $\bm{h}_i^{(0)}$ to each node $i$ in the graph $\mathcal{G}$. We adopt the GNN architecture proposed by~\citet{morris2019weisfeiler}, utilizing an $L$-layer GNN with updates at layer $l$ given by:
\begin{equation*}
    \begin{aligned}
        \bm{h}_i^{(l)} \coloneqq \sigma \left( \bm{\Phi}_1^{(l)} \bm{h}_i^{(l-1)} + \bm{\Phi}_2^{(l)} \sum_{j \in \mathcal{N}(i)} \bm{W}_{ji} \bm{h}_j^{(l-1)} \right),
    \end{aligned}
\end{equation*}
where $\sigma(\cdot)$ is an activation function, and $\bm{\Phi}_1^{(l)}$ and $\bm{\Phi}_2^{(l)}$ are the trainable parameters at layer $l$. This formulation facilitates efficient learning of node representations by leveraging both node features and the underlying graph structure. After processing through $L$ layers of GNN, we obtain the final output $\bm{H}^{(L)}_{\bm{\Phi}} \coloneqq [\bm{h}_1^{(L)}, \dots, \bm{h}_N^{(L)}] \in \mathbb{R}^{k \times N}$. A softmax activation function is applied in the last layer to ensure $\bm{H}^{(L)}_{\bm{\Phi}} \in \Delta_k^N$, making the final output feasible for $\overline{\bm{P}}$.

\textbf{\enquote{Pre-train + Fine-tune} Optimization.}
We propose a \enquote{pre-train + fine-tune} framework for learning the trainable weights of GNNs. Initially, the model is trained on a collection of pre-collected datasets to produce a pre-trained model. Subsequently, we fine-tune this pre-trained model for each specific testing instance. This approach equips the model with prior knowledge of graph structures during the pre-training phase, significantly reducing the overall solving time. Furthermore, it allows for out-of-distribution generalization due to the fine-tuning step.

In the pre-training phase, the trainable parameters $\bm{\Phi} \coloneqq (\bm{\Phi}_1^{(1)}, \bm{\Phi}_2^{(1)},\dots,\bm{\Phi}_1^{(L)}, \bm{\Phi}_2^{(L)})$  are optimized using the Adam optimizer with \textit{random initialization}, targeting the objective 
\begin{equation*}
    \begin{aligned}
        \min_{\bm{\Phi}} \quad \mathcal{L}_{\text{pre-training}}(\bm{\Phi}) \coloneqq \frac{1}{M} \sum_{m=1}^{M} f(\bm{H}_\Phi^{(L)}; \bm{W}_{\text{train}}^{(m)}),
    \end{aligned}
\end{equation*}
where $\mathcal{D} \coloneqq \{\bm{W}_{\text{train}}^{(1)}, \dots, \bm{W}_{\text{train}}^{(M)}\}$ represents the pre-training dataset. In the fine-tuning phase, for a problem instance $\bm{W}_{\text{test}}$, the Adam optimizer seeks to solve
\begin{equation*}
    \begin{aligned}
        \min_{\bm{\Phi}} \quad \mathcal{L}_{\text{fine-tuning}}(\bm{\Phi}) \coloneqq f(\bm{H}_\Phi^{(L)}; \bm{W}_{\text{test}}),
    \end{aligned}
\end{equation*}
initialized with the pre-trained parameters.

Moreover, to enable the GNN model to fully adapt to specific problem instances, the pre-training phase can be omitted, enabling the model to be directly trained and tested on the same instance. While this direct approach may necessitate more computational time, it often results in improved performance regarding the objective function. Consequently, users can choose to include a pre-training phase based on the specific requirements of their application scenarios.

\section{Experiments}

\subsection{Experimental Settings}

We compare the performance of \texttt{ROS} against traditional methods as well as L2O algorithms for solving the Max-$k$-Cut problem. Additionally, we assess the impact of the \enquote{Pre-train} stage in the GNN parametrization-based optimization. 
The source code is available at \url{https://github.com/NetSysOpt/ROS}.

\textbf{Baseline Algorithms.}
We denote our proposed algorithms by \texttt{ROS} and compare them against both traditional algorithms and L2O methods. 
When the pre-training step is skipped, we refer to our algorithm as \texttt{ROS-vanilla}.
The following traditional Max-$k$-Cut algorithms are considered as baselines:
(i)~\texttt{GW}~\cite{goemans1995improved}: an method with a $0.878$-approximation guarantee based on semi-definite relaxation;
(ii)~\texttt{BQP}~\cite{gui2018pci2}: a local search method designed for binary quadratic programs;
(iii)~\texttt{Genetic}~\cite{panxing2016pci}: a genetic algorithm specifically for Max-$k$-Cut problems;
(iv)~\texttt{MD}: a mirror descent algorithm that addresses the relaxed problem $\overline{\textbf{P}}$ with a convergence tolerance at $10^{-8}$ and adopts the same random sampling procedure;
(v)~\texttt{LPI}~\cite{goudet2024large}: an evolutionary algorithm featuring a large population organized across different islands;
(vi)~\texttt{MOH}~\cite{ma2017multiple}: a heuristic algorithm based on multiple operator heuristics, employing various distinct search operators within the search phase.
% (vii)~\texttt{Rank2}~\cite{burer2002rank}: a heuristic based on rank-2 relaxation.
For the L2O method, we primarily examine the state-of-the-art baseline algorithms:
(vii)~\texttt{PI-GNN}~\cite{schuetz2022combinatorial}: an unsupervised method for QUBO problems, which can model the weighted Max-Cut problem, delivering commendable performance.
% It is the first method to eliminate the dependence on large, labeled training datasets typically required by supervised learning approaches. 
(viii)~\texttt{ECO-DQN}~\cite{barrett2020exploratory}: a reinforcement L2O method introducing test-time exploratory refinement for Max-Cut problems.
(ix)~\texttt{ANYCSP}~\cite{tonshoff2022one}: an unsupervised GNN-based search heuristic for CSPs, which can model the unweighted Max-$k$-Cut problem, leveraging a compact graph representation and global search action with the default time limit of $180$ seconds.

\begin{table*}[ht]
\centering
\caption{Statistics of the training and testing datasets.}
\begin{tabular}{lccccc}
\toprule
     & Dataset & Graph Type & $N$ & \# Graphs & Weight Type \\
\midrule
Train & Random Regular Graph & regular & $100$ & $500$ & unweighted  \\
\midrule
\multirow{4}{*}{Test}
     & Random Regular Graph & regular & $100, 1{,}000, 10{,}000$ & $60$ & unweighted \\
     & Gset & random, planar, toroidal & $800\sim 20{,}000$ & $71$ & unweighted, weighted \\
     & COLOR & real-world & $74, 87, 138$ & $3$ & unweighted\\
     & Bitcoin-OTC & real-world & $5{,}881$ & $1$ & weighted \\
\bottomrule
\end{tabular}
\label{tab:data_stats}
\end{table*}

% We evaluate both unweighted and weighted versions, where the latter is derived by perturbing the original edge weights uniformly within $[-10\%,10\%]$, generating 10 perturbed instances per graph. 

\textbf{Datasets.} We conduct experiments on the following datasets.
\begin{itemize}
    \item \textbf{$r$-Random regular graphs}~\cite{schuetz2022combinatorial}: Each node has the same degree $r$. Edge weights are either $0$ or $1$. 
    \item \textbf{Gset}~\cite{ye2003}: A well-known Max-$k$-Cut benchmark comprising toroidal, planar, and random graphs with $800\sim20,000$ nodes and edge densities between $2\%$ and $6\%$. Edge weights are either $0$ or $\pm1$.
    \item \textbf{COLOR}~\cite{trick2003}: A collection of dense graphs derived from literary texts, where nodes represent characters and edges indicate co-occurrence. These graphs have large chromatic numbers ($\chi \approx 10$), making them suitable for Max-$k$-Cut. Edge weights are either $0$ or $1$.
    \item \textbf{Bitcoin-OTC}~\cite{kumar2016edge}: A real-world signed network with $5,881$ nodes and $35,592$ edges, weighted from $-10$ to $10$, capturing trust relationships among Bitcoin users.
\end{itemize}

The construction of the training and testing datasets is summarized in Table~\ref{tab:data_stats}. 
The training set consists of $500$ 3-regular, $500$ 5-regular graphs, and $500$ 7-regular graphs with $100$ nodes each, corresponding to the cases $k=2$, $k=3$, and $k=10$ respectively.
The test set of random regular graphs includes $20$ 3-regular and $20$ 5-regular graphs for each $k \in \{2, 3\}$, with node counts of $100$, $1{,}000$, and $10{,}000$. For the Gset benchmark, we evaluate both unweighted and weighted variants. 
The unweighted test set includes all Gset instances, with results reported in Tables~\ref{tab:3} and~\ref{tab:4} in Appendix~\ref{app:gset}. 
For the weighted variant, we generate perturbations of the four largest Gset graphs (G70, G72, G77, G81) by multiplying each edge weight by $\sigma \sim \mathcal{U}[l, u]$, creating 10 perturbed instances per graph. We examine three distinct perturbation regimes: (i) mild perturbations ($[0.9, 1.1]$), (ii) moderate variations ($[0, 10]$), and (iii) extreme modifications ($[0, 100]$). The moderate perturbation results ($[0, 10]$) are presented in Table~\ref{tab:weighted_gset}, with the remaining cases available in Appendix~\ref{app:full_weighted_gset}. Additionally, we evaluate performance on three COLOR benchmark instances: \texttt{anna}, \texttt{david}, and \texttt{huck}.

\textbf{Model Settings.} 
{\texttt{ROS} is designed as a two-layer GNN, with both the input and hidden dimensions set to $100$. To address the issue of gradient vanishing, we apply graph normalization as proposed by \citet{cai2021graphnorm}.
The \texttt{ROS} model is pre-training using Adam with a learning rate of $10^{-2}$ for one epoch.
During fine-tuning, the model is further optimized using the same Adam optimizer and learning rate, applying early stopping with a tolerance of $10^{-2}$ and patience of $100$ iterations. Training terminates if no improvement is observed. Finally, in the random sampling stage, we execute Algorithm \ref{alg:random} for $T=100$ trials and return the best solution. }

\textbf{Evaluation Configuration.} All our experiments were conducted on an NVIDIA RTX 3090 GPU, using PyTorch 2.2.0.

\subsection{{Performance Comparison against Baselines} }

\subsubsection{ {Computational Time}\label{sec:time}}

\begin{figure}[t]
    \centering
    \begin{subfigure}{0.45\textwidth}
        \centering      \includegraphics[width=\textwidth]{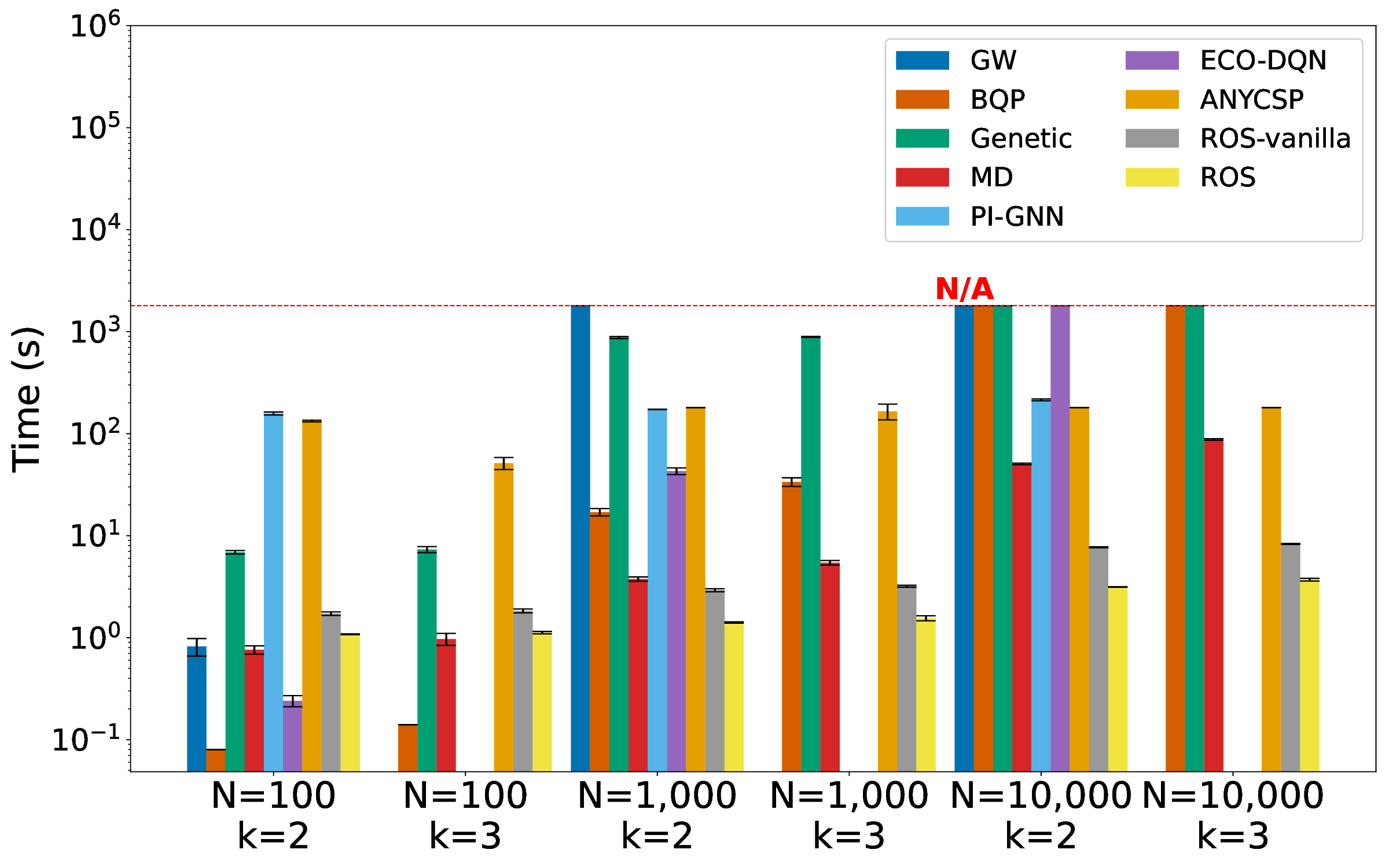}
        \caption{Random regular graph \label{fig:time1}}
    \end{subfigure}
    \hfill
    \begin{subfigure}{0.45\textwidth}
        \centering
\includegraphics[width=\textwidth]{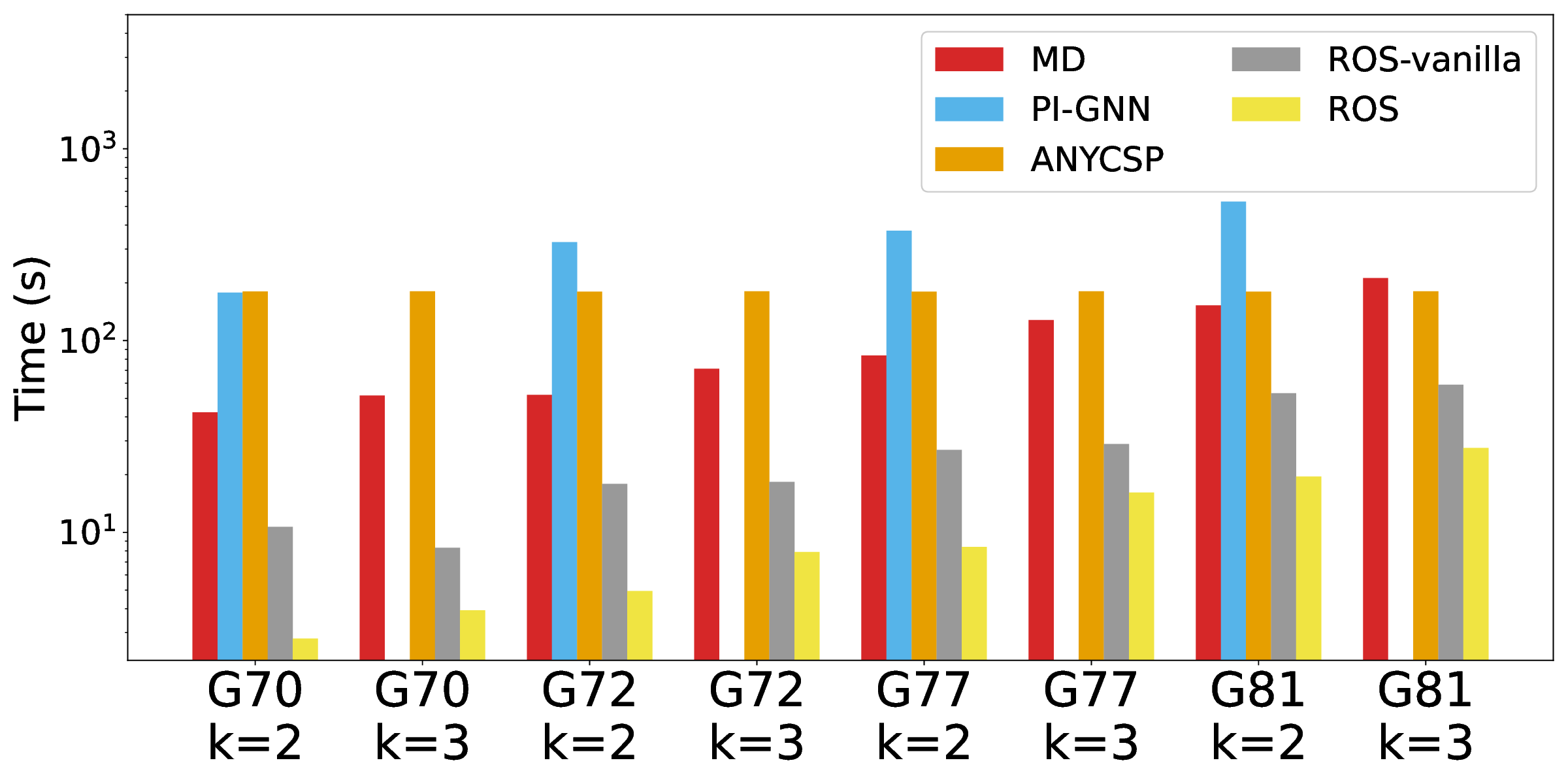}
    \caption{Weighted Gset with perturbation ratio $[0, 10]$\label{fig:time2}}
    \end{subfigure}
    \hfill
    \begin{subfigure}{0.45\textwidth}
        \centering
\includegraphics[width=\textwidth]{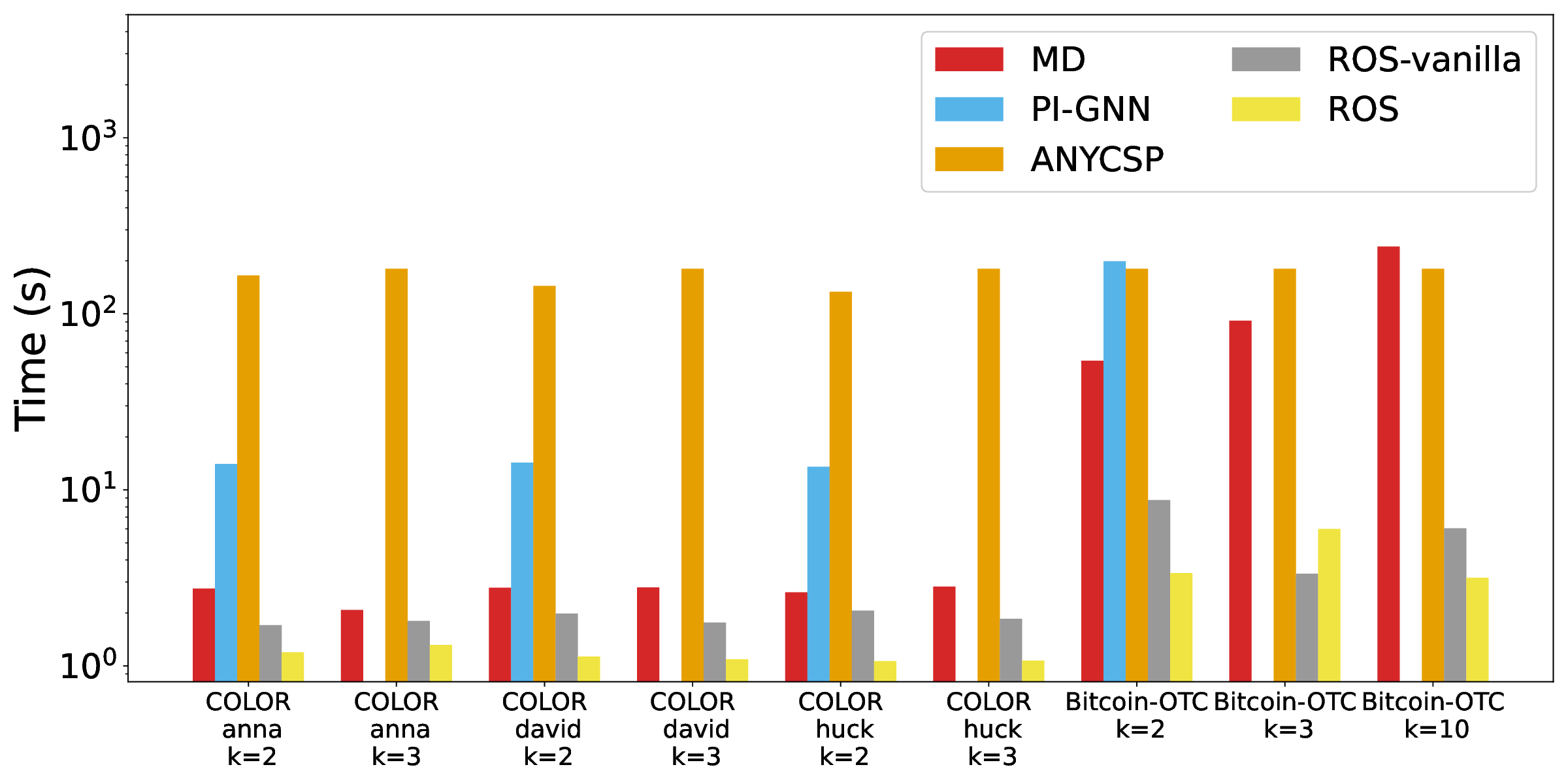}
    \caption{COLOR datasets and Bitcoin-OTC datasets\label{fig:time3}}
    \end{subfigure}
    \caption{The computational time comparison of Max-$k$-Cut problems.}
    \label{fig:time}
\end{figure}

\begin{table*}[t]
\centering
\caption{Cut value comparison of Max-$k$-Cut problems on random regular graphs.}
{\small %
\begin{tabular}{c c@{\hskip 0.05in}c c@{\hskip 0.05in}c c@{\hskip 0.05in}c}
\midrule
\multirow{2}{*}{Methods} & \multicolumn{2}{c}{\texttt{N=100}}            & \multicolumn{2}{c}{\texttt{N=1,000}}              & \multicolumn{2}{c}{\texttt{N=10,000}}              \\ 
 \cmidrule(l){2-3}  \cmidrule(l){4-5}  \cmidrule(l){6-7}
           & $k=2$         & $k=3$   &$k=2$& $k=3$ &$k=2$& $k=3$ \\ \midrule
\texttt{GW}& $130.20_{\pm2.79}$& {\bf --} & N/A   & {\bf --}  & N/A                 & {\bf --}                  \\
\texttt{BQP}& $131.55_{\pm 2.42}$ &$239.70_{\pm 1.82}$&$1324.45_{\pm 6.34}$& $2419.15_{\pm 6.78}$ & N/A  & N/A          \\
\texttt{Genetic}& $127.55_{\pm 2.82}$& $235.50_{\pm 3.15}$& $1136.65_{\pm 10.37}$  & $2130.30_{\pm 8.49}$    & N/A                & N/A                 \\
\texttt{MD}              & $127.20_{\pm 2.16}$  & $235.50_{\pm 3.29}$        & $ 1250.35_{\pm 11.21}$ & $ 2344.85_{\pm 9.86}$       & $12428.85_{\pm 26.13}$ &$23341.20_{\pm 32.87}$   \\ 
\texttt{PI-GNN}          & $122.95_{\pm 3.83}$ & {\bf --}    & $1210.45_{\pm 44.56}$    & {\bf --}   & $12655.05_{\pm 94.25}$  & {\bf --}    \\ 
\texttt{ECO-DQN}          & $135.60_{\pm 1.53}$ & {\bf --}    & $1366.20_{\pm 5.20}$    & {\bf --}   & N/A  & {\bf --}    \\ 
\texttt{ANYCSP}          & $131.65_{\pm 3.35}$ &  $247.90_{\pm 0.89}$   & $1366.05_{\pm 5.25}$    & $2494.50_{\pm 2.99}$   & $13692.35_{\pm 11.27}$  & $24929.80_{\pm 7.53}$    \\ \hline
\texttt{ROS-vanilla} & $132.80_{\pm 1.99}$ & $243.20_{\pm 1.80}$      & $1322.95_{\pm 6.57}$   & $2443.9_{\pm 4.10}$       & $13239.80_{\pm 14.71}$ & $24413.30_{\pm 16.02}$       \\
\texttt{ROS} & $128.20_{\pm 2.82}$ & $240.30_{\pm 2.59}$      & $1283.75_{\pm 6.89}$   & $2405.75_{\pm 5.72}$       & $12856.85_{\pm 26.50}$ & $24085.95_{\pm 21.88}$       \\ \bottomrule
\end{tabular}%
}
\label{tab:reg3}
\end{table*}

\begin{table*}[t]
\centering
\caption{Cut value comparison of Max-$k$-Cut problems on weighted Gset instances, where the noise factor $\sigma\sim[0, 10]$.}
\small{
% \resizebox{\textwidth}{!}{%
\begin{tabular}{p{1.8cm}<{\centering}p{1.1cm}<{\centering}p{1.1cm}<{\centering}p{1.1cm}<{\centering}p{1.1cm}<{\centering}p{1.1cm}<{\centering}p{1.1cm}<{\centering}p{1.1cm}<{\centering}p{1.1cm}<{\centering}}
\midrule
\multirow{2}{*}{Methods}                      & \multicolumn{2}{c}{\texttt{G70 (N=10,000)}}              & \multicolumn{2}{c}{\texttt{G72 (N=10,000)}}                 & \multicolumn{2}{c}{\texttt{G77 (N=14,000)}}                 & \multicolumn{2}{c}{\texttt{G81 (N=20,000)}}                \\\cmidrule(l){2-3}\cmidrule(l){4-5}\cmidrule(l){6-7}\cmidrule(l){8-9}
                           & $k=2$   &$k=3$&$k=2$   &$k=3$& $k=2$   &$k=3$& $k=2$ &$k=3$   \\ \midrule
  \texttt{GW}&N/A&{\bf --}&N/A&{\bf --}&N/A&{\bf --}&N/A&{\bf --}                             \\ 
  \texttt{BQP}&N/A&N/A&N/A&N/A&N/A&N/A&N/A&N/A \\ 
  \texttt{Genetic}&N/A&N/A&N/A&N/A&N/A&N/A&N/A&N/A \\ 
\texttt{MD}&$45490.21$&$49615.85$&$33449.49$&$38798.78$&$47671.94$&$55147.26$&$67403.00$&$78065.07$                           \\ 
\texttt{PI-GNN}&$44275.72$&\textbf{--}&$31469.65$&\textbf{--}&$44359.72$&\textbf{--}&$62439.97$&\textbf{--}                            \\ 
     \texttt{ECO-DQN}&N/A&{\bf --}&N/A&{\bf --}&N/A&{\bf --}&N/A&{\bf --}                             \\ 
\texttt{ANYCSP}  & $46420.48$&$48831.32$& $-280.74$&$-208.01$& $845.72$& $988.96$ & $-13.52$& $271.01$ \\ \midrule
\texttt{ROS-vanilla}&$47140.07$ & $49826.90$ &  $36697.11$ & $42067.80$ & $52226.53$ &  $59636.36$ &$74051.42$& $84498.44$\\
\texttt{ROS}     & $46707.60$&$49813.45$& $35733.11$&$40987.92$& $50790.44$&$58253.31$&    $72057.24$&$82450.68$  \\\midrule
\end{tabular}%
\label{tab:weighted_gset}
}
\end{table*}

We evaluated \texttt{ROS} against seven baseline algorithms: \texttt{GW}, \texttt{BQP}, \texttt{Genetic}, \texttt{MD}, \texttt{PI-GNN}, \texttt{ECO-DQN}, and \texttt{ANYCSP} on random regular graphs, comparing computational time for both Max-Cut and Max-$3$-Cut tasks. Experiments covered three problem scales: $N = 100$, $N = 1{,}000$, and $N = 10{,}000$, with results shown in Figure~\ref{fig:time1}.
For larger instances, Figure~\ref{fig:time2} compares the scalable methods (\texttt{MD}, \texttt{ANYCSP}, and \texttt{PI-GNN}) on weighted Gset graphs ($N \geq 10{,}000$) with edge weight perturbations in $[0, 10]$. Figure~\ref{fig:time3} extends this comparison to real-world networks (COLOR and Bitcoin-OTC graphs). 
Instances marked \enquote{N/A} indicate timeout failures (30-minute limit). Complete results for unweighted Gset benchmarks, including comparisons with state-of-the-art methods \texttt{LPI} and \texttt{MOH}, are provided in Tables~\ref{tab:3} and~\ref{tab:4} (Appendix~\ref{app:gset}).

The results depicted in Figure~\ref{fig:time1} indicate that \texttt{ROS} efficiently solves all problem instances within seconds, even for large problem sizes of \( N = 10,000 \). 
In terms of baseline performance, the approximation algorithm \texttt{GW} performs efficiently on instances with \( N = 100 \), but it struggles with larger sizes due to the substantial computational burden associated with solving the underlying semi-definite programming problem. 
Heuristic methods such as \texttt{BQP} and \texttt{Genetic} can manage cases up to \( N = 1,000 \) in a few hundred seconds, yet they fail to solve larger instances with \( N = 10,000 \) because of the high computational cost of each iteration. 
Notably, \texttt{MD} is the only traditional method capable of solving large instances within a reasonable time frame; however,  when \( N \) reaches \( 10,000 \), the computational time for \texttt{MD} approaches $15$ times that of \texttt{ROS}. 
Regarding L2O methods, \texttt{PI-GNN} necessitates retraining and prediction for each instance, with test times exceeding dozens of seconds even for \( N = 100 \). 
\texttt{ECO-DQN} relies on expensive GNNs at each decision step and can not scale to large problem sizes of $N=10,000$. 
\texttt{ANYCSP} needs hundreds of seconds even for $N=100$ due to the global search operation and long sampling trajectory. 
In contrast, \texttt{ROS} solves these large instances in merely a few seconds throughout the experiments, requiring only \( 10\% \) of the computational time utilized by other L2O baselines. 
Figure~\ref{fig:time2} and Figure~\ref{fig:time3} illustrate the results for the weighted Gset benchmark and real-world datasets, respectively, where \texttt{ROS} efficiently solves the largest instances in just a few seconds, while other methods take tens to hundreds of seconds for equivalent tasks. Remarkably, \texttt{ROS} utilizes only about \( 1\% \) of the computational time required by \texttt{PI-GNN}.

\begin{table*}[h]
\centering
\caption{Cut value comparison of Max-$k$-Cut problems on COLOR datasets and Bitcoin-OTC Datasets.}
\begin{tabular}{cccccccccc}
\midrule
\multirow{2}{*}{Methods}                      & \multicolumn{2}{c}{COLOR anna} & \multicolumn{2}{c}{COLOR david}& \multicolumn{2}{c}{COLOR huck}   & \multicolumn{3}{c}{Bitcoin-OTC}                \\\cmidrule(l){2-7}\cmidrule(l){8-10}
                           & $k=2$   &$k=3$&$k=2$   &$k=3$& $k=2$   &$k=3$& $k=2$ &$k=3$ & $k=10$ \\ \midrule
\texttt{MD}& $339$        & $421$ & $259$        & $329$        & $184$        & $242$&$39076$&$47595$ &$53563$\\ 
\texttt{PI-GNN}&$279$&\textbf{--}&$228$&\textbf{--}&$166$&\textbf{--}&$37216$&\textbf{--}&\textbf{--} \\ 
\texttt{ANYCSP}  & $330$&$423$& $263$&$328$& $166$& $139$ & $10431$& $14265$ & $19372$\\ \midrule
\texttt{ROS-vanilla}&$351$ & $429$ &  $266$ & $336$ & $191$ &  $246$ &$40576$& $48214$ & $53758$\\
\texttt{ROS}     & $351$&$423$& $266$&$324$& $191$&$242$&    $39850$&$48980$ & $53778$ \\\midrule
\end{tabular}
\label{tab:real_world}
\end{table*}

\subsubsection{Cut Value \label{exp:obj}}

% \texttt{ANYCSP}  & $5198.87_{\pm69.76}$&  & $-15.57_{\pm57.88}$& & $81.76_{\pm69.97}$&    & $33.49_{\pm50.31}$&  \\ 
% \texttt{ROS}     & $\mathbf{8941.80_{\pm17.79}}$&  & $\mathbf{6165.62_{\pm50.81}}$&  & $\mathbf{8737.59_{\pm114.24}}$&    & $\mathbf{12325.85_{\pm87.98}}$& 

We evaluate \texttt{ROS}'s performance on random regular graphs, the Gset benchmark, and real-world datasets, measuring solution quality for Problem~(\ref{eq:max-k-cut_formulation}). Results appear in Tables~\ref{tab:reg3} (random graphs), \ref{tab:weighted_gset} (weighted Gset), and \ref{tab:real_world} (real-world data), where \enquote{\bfseries --} denotes methods incompatible with Max-$k$-Cut problems.

The results demonstrate that \texttt{ROS} consistently produces high-quality solutions for both $k = 2$ and $k = 3$ across all scales. While \texttt{GW} performs well for Max-Cut ($k=2$) at $N=100$, it fails to generalize to arbitrary $k$. Traditional methods like \texttt{BQP} and \texttt{Genetic} support $k=3$ but often converge to suboptimal solutions. Although \texttt{MD} handles general $k$, it consistently underperforms \texttt{ROS}. Among learning-based methods, \texttt{PI-GNN} proves unsuitable for $k=3$ due to QUBO incompatibility and unreliable heuristic rounding, while \texttt{ECO-DQN} lacks $k=3$ support entirely. While \texttt{ANYCSP} achieves strong results on unweighted graphs, it cannot process weighted instances.
These experiments collectively show that \texttt{ROS} offers superior generalizability and robustness for weighted Max-$k$-Cut tasks, outperforming both traditional and learning-based approaches in solution quality and flexibility.

To further assess \texttt{ROS}'s scalability, we conduct comprehensive benchmarking against scalable baselines using challenging real-world datasets, including the COLOR and Bitcoin-OTC networks. The results in Table~\ref{tab:real_world} demonstrate that both \texttt{ROS} and its simplified variant \texttt{ROS-vanilla} consistently outperform competing methods across most experimental settings. This performance advantage is particularly pronounced for the weighted Bitcoin-OTC instances, where our approach achieves superior solution quality while maintaining computational efficiency.

\subsection{Effect of the \enquote{Pre-train} Stage in \texttt{ROS}}

To evaluate the impact of the pre-training stage in \texttt{ROS}, we compared it with \texttt{ROS-vanilla}, which omits pre-training (see Section~\ref{sec:GNN_optimization}). We assessed both methods based on cut values and computational time. Figure~\ref{fig:ratio} illustrates the ratios of these metrics between \texttt{ROS-vanilla} and \texttt{ROS}. In this figure, the horizontal axis represents the problem instances, while the left vertical axis (green bars) displays the ratio of objective function values, and the right vertical axis (red curve) indicates the ratio of computational times.

As shown in Figure~\ref{fig:ratio1}, \texttt{ROS-vanilla} achieves higher objective function values in most settings on the random regular graphs; however, its computational time is approximately \(1.5\) times greater than that of \texttt{ROS}. Thus, \texttt{ROS} demonstrates a faster solving speed compared to \texttt{ROS-vanilla}. Similarly, in experiments conducted on the Gset benchmark (Figure~\ref{fig:ratio2}), \texttt{ROS} reduces computational time by around \(40\%\) while maintaining performance comparable to that of \texttt{ROS-vanilla}. Notably, in the Max-3-Cut problem for the largest instance, \texttt{G81}, \texttt{ROS} effectively halves the solving time, showcasing the significant acceleration effect of pre-training. It is worth mentioning that the \texttt{ROS} model was pre-trained on random regular graphs with \(N = 100\) and generalized well to regular graphs with \(N = 1{,}000\) and \(N = 10{,}000\), as well as to Gset problem instances of varying sizes and types. This illustrates \texttt{ROS}'s capability to generalize and accelerate the solving of large-scale problems across diverse graph types and sizes, emphasizing the strong out-of-distribution generalization afforded by pre-training.

In summary, while \texttt{ROS-vanilla} achieves slightly higher objective function values on individual instances, it requires longer solving times and struggles to generalize to other problem instances. This observation highlights the trade-off between a model's ability to generalize and its capacity to fit specific instances. Specifically, a model that fits individual instances exceptionally well may fail to generalize to new data, resulting in longer solving times. Conversely, a model that generalizes effectively may exhibit slightly weaker performance on specific instances, leading to a marginal decrease in objective function values. Therefore, the choice between these two training modes should be guided by the specific requirements of the application.

\begin{figure}[t]
    \centering
    \begin{subfigure}{0.43\textwidth} % 确保单位为 \linewidth
        \centering
        \includegraphics[width=\textwidth]{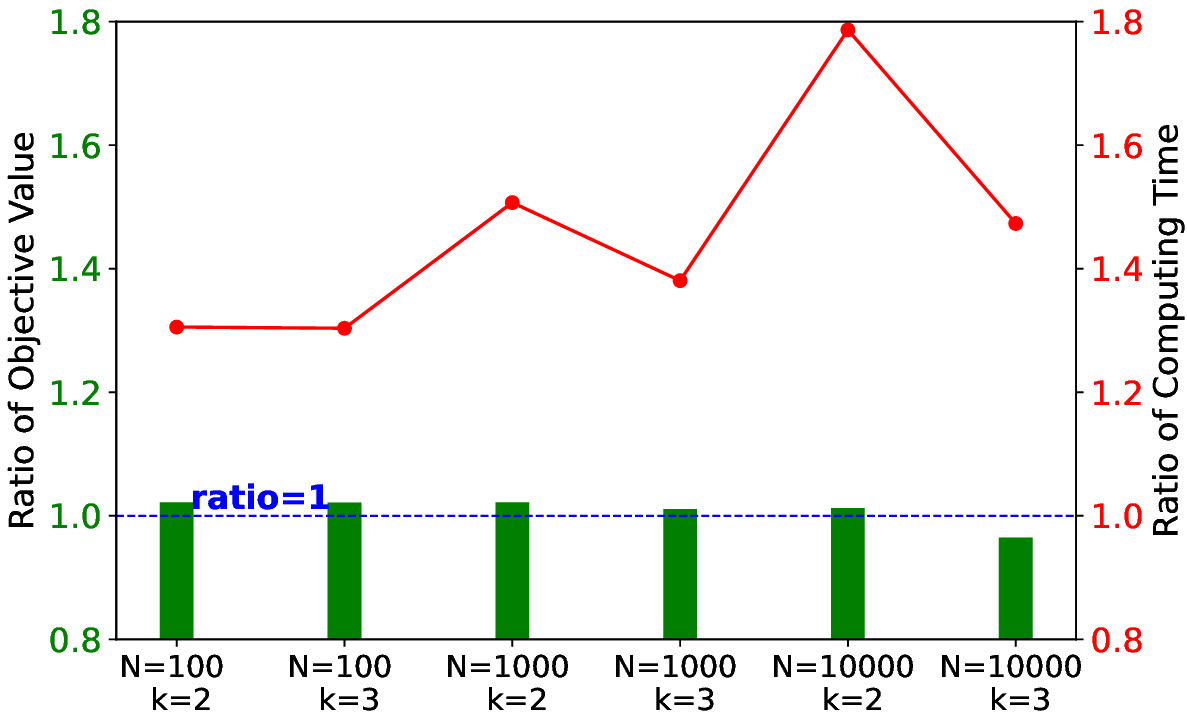}
        \caption{Random regular graph \label{fig:ratio1}}
    \end{subfigure}
    \hfill
    \begin{subfigure}{0.43\textwidth} % 确保单位为 \linewidth
        \centering
        \includegraphics[width=\textwidth]{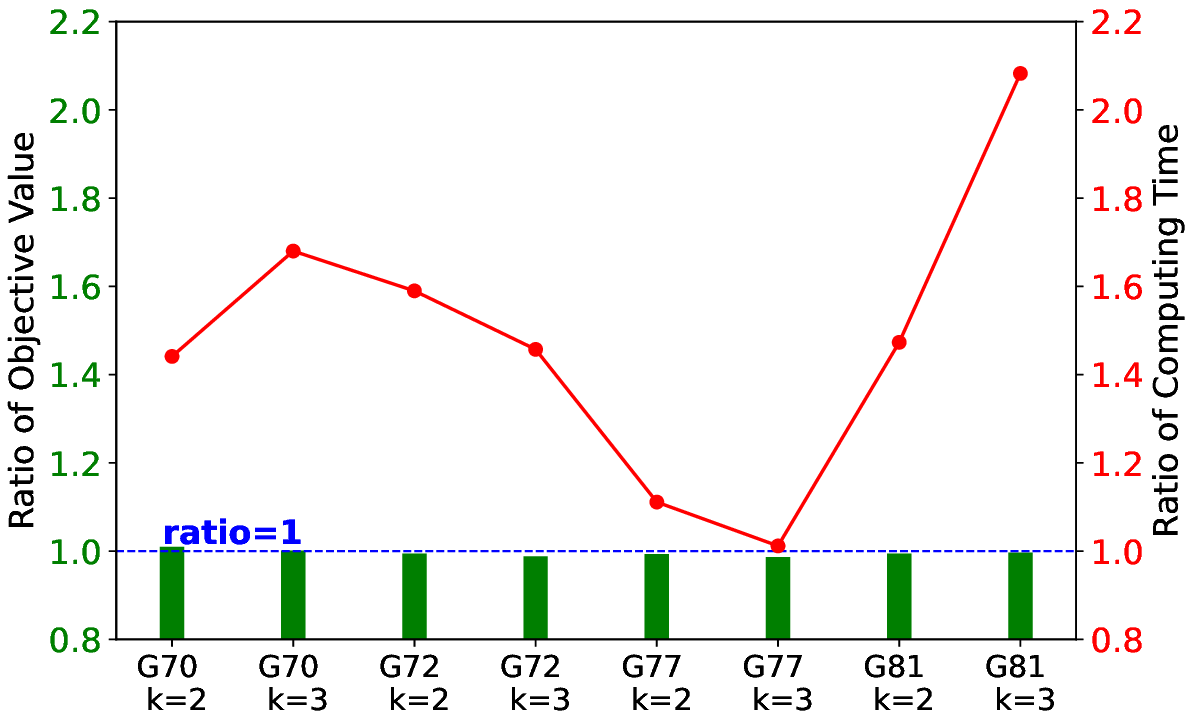}
        \caption{Gset \label{fig:ratio2}}
    \end{subfigure}
    \caption{The ratio of computational time and cut value comparison between $\texttt{ROS-vanilla}$ and $\texttt{ROS}$.}
    \label{fig:ratio}
\end{figure}

\section{Conclusions}

In this paper, we propose ROS, an efficient method for addressing the Max-$k$-Cut problem with arbitrary edge weights. 
Our approach begins by relaxing the constraints of the original discrete problem to probabilistic simplices. 
To effectively solve this relaxed problem, we propose an optimization algorithm based on GNN parametrization and incorporate transfer learning by leveraging pre-trained GNNs to warm-start the training process. 
After resolving the relaxed problem, we present a novel random sampling algorithm that maps the continuous solution back to a discrete form. By integrating geometric landscape analysis with statistical theory, we establish the consistency of function values between the continuous and discrete solutions. 
Experiments conducted on random regular graphs, the Gset benchmark, and real-world datasets demonstrate that our method is highly efficient for solving large-scale Max-$k$-Cut problems, requiring only a few seconds, even for instances with tens of thousands of variables. Furthermore, it exhibits robust generalization capabilities across both in-distribution and out-of-distribution instances, highlighting its effectiveness for large-scale optimization tasks.
Exploring other sampling algorithms to further boost ROS performance is a future research direction. Moreover, the ROS framework with theoretical insights could be potentially extended to other graph-related combinatorial problems, and this direction is also worth investigating as future work.

%\newpage

\section*{Impact Statement}
This paper presents work whose goal is to advance the field of Machine Learning. There are many potential societal consequences of our work, none of which we feel must be specifically highlighted here.

\section*{Acknowledgement}
This work was supported by the National Key R\&D Program of China under grant 2022YFA1003900. Ye Xue  acknowledges support from the National Natural Science Foundation of China (Grant No. 62301334), the Guangdong Major Project of Basic and
Applied Basic Research (No. 2023B0303000001), Akang Wang also acknowledges support from the National Natural Science Foundation of China (Grant No. 12301416), the Guangdong Basic and Applied Basic Research Foundation (Grant No. 2024A1515010306).

% In the unusual situation where you want a paper to appear in the
% references without citing it in the main text, use \nocite

\bibliography{example_paper}

\begin{thebibliography}{43}
\providecommand{\natexlab}[1]{#1}
\providecommand{\url}[1]{\texttt{#1}}
\expandafter\ifx\csname urlstyle\endcsname\relax
  \providecommand{\doi}[1]{doi: #1}\else
  \providecommand{\doi}{doi: \begingroup \urlstyle{rm}\Url}\fi

\bibitem[Andrade et~al.(2024)Andrade, Pessoa, and
  Stawiarski]{andrade2022physical}
Andrade, C.~E., Pessoa, L.~S., and Stawiarski, S.
\newblock The physical cell identity assignment problem: A practical
  optimization approach.
\newblock \emph{IEEE Transactions on Evolutionary Computation}, 28\penalty0
  (2):\penalty0 282--292, 2024.

\bibitem[Bach(2013)]{bach2013learning}
Bach, F.
\newblock \emph{Learning with Submodular Functions: A Convex Optimization
  Perspective}.
\newblock Now Publishers Inc., Hanover, MA, USA, 2013.
\newblock ISBN 1601987560.

\bibitem[Barrett et~al.(2020)Barrett, Clements, Foerster, and
  Lvovsky]{barrett2020exploratory}
Barrett, T., Clements, W., Foerster, J., and Lvovsky, A.
\newblock Exploratory combinatorial optimization with reinforcement learning.
\newblock \emph{Proceedings of the AAAI Conference on Artificial Intelligence},
  34\penalty0 (04):\penalty0 3243--3250, 2020.

\bibitem[Barrett et~al.(2022)Barrett, Parsonson, and
  Laterre]{barrett2022learning}
Barrett, T.~D., Parsonson, C.~W., and Laterre, A.
\newblock Learning to solve combinatorial graph partitioning problems via
  efficient exploration.
\newblock \emph{arXiv preprint arXiv:2205.14105}, 2022.

\bibitem[Bengio et~al.(2021)Bengio, Lodi, and Prouvost]{bengio2021machine}
Bengio, Y., Lodi, A., and Prouvost, A.
\newblock Machine learning for combinatorial optimization: A methodological
  tour d’horizon.
\newblock \emph{European Journal of Operational Research}, 290\penalty0
  (2):\penalty0 405--421, 2021.
\newblock ISSN 0377-2217.

\bibitem[Bhowmick et~al.(2024)Bhowmick, Kosan, Huang, Singh, and
  Medya]{bhowmick2024dgcluster}
Bhowmick, A., Kosan, M., Huang, Z., Singh, A., and Medya, S.
\newblock Dgcluster: A neural framework for attributed graph clustering via
  modularity maximization.
\newblock \emph{Proceedings of the AAAI Conference on Artificial Intelligence},
  38\penalty0 (10):\penalty0 11069--11077, 2024.

\bibitem[Burer et~al.(2002)Burer, Monteiro, and Zhang]{burer2002rank}
Burer, S., Monteiro, R. D.~C., and Zhang, Y.
\newblock Rank-two relaxation heuristics for max-cut and other binary quadratic
  programs.
\newblock \emph{SIAM Journal on Optimization}, 12\penalty0 (2):\penalty0
  503--521, 2002.

\bibitem[Cai et~al.(2021)Cai, Luo, Xu, He, Liu, and Wang]{cai2021graphnorm}
Cai, T., Luo, S., Xu, K., He, D., Liu, T.-Y., and Wang, L.
\newblock Graphnorm: A principled approach to accelerating graph neural network
  training.
\newblock In \emph{Proceedings of the 38th International Conference on Machine
  Learning}, volume 139 of \emph{Proceedings of Machine Learning Research},
  pp.\  1204--1215. PMLR, 18--24 Jul 2021.

\bibitem[Carlson \& Nemhauser(1966)Carlson and
  Nemhauser]{carlson1966scheduling}
Carlson, R. and Nemhauser, G.~L.
\newblock Scheduling to minimize interaction cost.
\newblock \emph{Operations Research}, 14\penalty0 (1):\penalty0 52--58, 1966.

\bibitem[Chen et~al.(2020)Chen, Chen, Du, Wei, and Chen]{chen2020heuristic}
Chen, M., Chen, Y., Du, Y., Wei, L., and Chen, Y.
\newblock Heuristic algorithms based on deep reinforcement learning for
  quadratic unconstrained binary optimization.
\newblock \emph{Knowledge-Based Systems}, 207:\penalty0 106366, 2020.
\newblock ISSN 0950-7051.

\bibitem[Chen et~al.(2024)Chen, Liu, and Yin]{chen2024learning}
Chen, X., Liu, J., and Yin, W.
\newblock Learning to optimize: A tutorial for continuous and mixed-integer
  optimization.
\newblock \emph{Science China Mathematics}, pp.\  1--72, 2024.

\bibitem[Coja-Oghlan et~al.(2022)Coja-Oghlan, Loick, Mezei, and
  Sorkin]{coja2022ising}
Coja-Oghlan, A., Loick, P., Mezei, B.~F., and Sorkin, G.~B.
\newblock The ising antiferromagnet and max cut on random regular graphs.
\newblock \emph{SIAM Journal on Discrete Mathematics}, 36\penalty0
  (2):\penalty0 1306--1342, 2022.

\bibitem[Cook et~al.(2019)Cook, Zhao, Sato, Hiromoto, and Tan]{cook2019gpu}
Cook, C., Zhao, H., Sato, T., Hiromoto, M., and Tan, S. X.-D.
\newblock Gpu-based ising computing for solving max-cut combinatorial
  optimization problems.
\newblock \emph{Integration}, 69:\penalty0 335--344, 2019.
\newblock ISSN 0167-9260.

\bibitem[de~Klerk et~al.(2004)de~Klerk, Pasechnik, and
  Warners]{de2004approximate}
de~Klerk, E., Pasechnik, D.~V., and Warners, J.~P.
\newblock On approximate graph colouring and max-k-cut algorithms based on the
  $\theta$-function.
\newblock \emph{Journal of Combinatorial Optimization}, 8:\penalty0 267--294,
  2004.

\bibitem[De~Simone et~al.(1995)De~Simone, Diehl, J{\"u}nger, Mutzel, Reinelt,
  and Rinaldi]{de1995exact}
De~Simone, C., Diehl, M., J{\"u}nger, M., Mutzel, P., Reinelt, G., and Rinaldi,
  G.
\newblock Exact ground states of ising spin glasses: New experimental results
  with a branch-and-cut algorithm.
\newblock \emph{Journal of Statistical Physics}, 80:\penalty0 487--496, 1995.

\bibitem[Eisenbl{\"a}tter(2002)]{eisenblatter2002semidefinite}
Eisenbl{\"a}tter, A.
\newblock The semidefinite relaxation of the k-partition polytope is strong.
\newblock In \emph{Integer Programming and Combinatorial Optimization}, pp.\
  273--290, Berlin, Heidelberg, 2002. Springer Berlin Heidelberg.

\bibitem[Frieze \& Jerrum(1997)Frieze and Jerrum]{frieze1997improved}
Frieze, A. and Jerrum, M.
\newblock Improved approximation algorithms for max k-cut and max bisection.
\newblock \emph{Algorithmica}, 18\penalty0 (1):\penalty0 67--81, 1997.

\bibitem[Garvardt et~al.(2023)Garvardt, Grüttemeier, Komusiewicz, and
  Morawietz]{garvardt2023parameterized}
Garvardt, J., Grüttemeier, N., Komusiewicz, C., and Morawietz, N.
\newblock Parameterized local search for max c-cut.
\newblock In Elkind, E. (ed.), \emph{Proceedings of the Thirty-Second
  International Joint Conference on Artificial Intelligence, {IJCAI-23}}, pp.\
  5586--5594. International Joint Conferences on Artificial Intelligence
  Organization, 8 2023.
\newblock Main Track.

\bibitem[Gasse et~al.(2022)Gasse, Bowly, Cappart, Charfreitag, Charlin,
  Ch{\'e}telat, Chmiela, Dumouchelle, Gleixner, Kazachkov,
  et~al.]{gasse2022machine}
Gasse, M., Bowly, S., Cappart, Q., Charfreitag, J., Charlin, L., Ch{\'e}telat,
  D., Chmiela, A., Dumouchelle, J., Gleixner, A., Kazachkov, A.~M., et~al.
\newblock The machine learning for combinatorial optimization competition
  (ml4co): Results and insights.
\newblock In \emph{NeurIPS 2021 competitions and demonstrations track}, pp.\
  220--231. PMLR, 2022.

\bibitem[Ghaddar et~al.(2011)Ghaddar, Anjos, and Liers]{ghaddar2011branch}
Ghaddar, B., Anjos, M.~F., and Liers, F.
\newblock A branch-and-cut algorithm based on semidefinite programming for the
  minimum k-partition problem.
\newblock \emph{Annals of Operations Research}, 188\penalty0 (1):\penalty0
  155--174, 2011.

\bibitem[Goemans \& Williamson(1995)Goemans and
  Williamson]{goemans1995improved}
Goemans, M.~X. and Williamson, D.~P.
\newblock Improved approximation algorithms for maximum cut and satisfiability
  problems using semidefinite programming.
\newblock \emph{Journal of the ACM (JACM)}, 42\penalty0 (6):\penalty0
  1115--1145, 1995.

\bibitem[Goudet et~al.(2024)Goudet, Goëffon, and Hao]{goudet2024large}
Goudet, O., Goëffon, A., and Hao, J.-K.
\newblock A large population island framework for the unconstrained binary
  quadratic problem.
\newblock \emph{Computers \& Operations Research}, 168:\penalty0 106684, 2024.
\newblock ISSN 0305-0548.

\bibitem[Gui et~al.(2019)Gui, Jiang, and Gao]{gui2018pci2}
Gui, J., Jiang, Z., and Gao, S.
\newblock Pci planning based on binary quadratic programming in lte/lte-a
  networks.
\newblock \emph{IEEE Access}, 7:\penalty0 203--214, 2019.

\bibitem[Hojny et~al.(2021)Hojny, Joormann, L{\"u}then, and
  Schmidt]{hojny2021mixed}
Hojny, C., Joormann, I., L{\"u}then, H., and Schmidt, M.
\newblock Mixed-integer programming techniques for the connected max-k-cut
  problem.
\newblock \emph{Mathematical Programming Computation}, 13\penalty0
  (1):\penalty0 75--132, 2021.

\bibitem[Karalias \& Loukas(2020)Karalias and Loukas]{karalias2020erdos}
Karalias, N. and Loukas, A.
\newblock Erdos goes neural: an unsupervised learning framework for
  combinatorial optimization on graphs.
\newblock In \emph{Advances in Neural Information Processing Systems},
  volume~33, pp.\  6659--6672. Curran Associates, Inc., 2020.

\bibitem[Karp(1972)]{karp2010reducibility}
Karp, R.~M.
\newblock \emph{Reducibility among Combinatorial Problems}, pp.\  85--103.
\newblock Springer US, Boston, MA, 1972.
\newblock ISBN 978-1-4684-2001-2.

\bibitem[Khalil et~al.(2017)Khalil, Dai, Zhang, Dilkina, and
  Song]{khalil2017learning}
Khalil, E., Dai, H., Zhang, Y., Dilkina, B., and Song, L.
\newblock Learning combinatorial optimization algorithms over graphs.
\newblock In \emph{Advances in Neural Information Processing Systems},
  volume~30. Curran Associates, Inc., 2017.

\bibitem[Kumar et~al.(2016)Kumar, Spezzano, Subrahmanian, and
  Faloutsos]{kumar2016edge}
Kumar, S., Spezzano, F., Subrahmanian, V.~S., and Faloutsos, C.
\newblock Edge weight prediction in weighted signed networks.
\newblock In \emph{2016 IEEE 16th International Conference on Data Mining
  (ICDM)}, pp.\  221--230, 2016.

\bibitem[Li \& Wang(2016)Li and Wang]{panxing2016pci}
Li, P. and Wang, J.
\newblock Pci planning method based on genetic algorithm in lte network.
\newblock \emph{Telecommunications Science}, 32\penalty0 (3):\penalty0 2016082,
  2016.

\bibitem[Ly et~al.(2023)Ly, Sawhney, and Chugunova]{ly2023data}
Ly, A., Sawhney, R., and Chugunova, M.
\newblock Data clustering and visualization with recursive goemans-williamson
  maxcut algorithm.
\newblock In \emph{2023 International Conference on Computational Science and
  Computational Intelligence (CSCI)}, pp.\  496--500. IEEE, 2023.

\bibitem[Ma \& Hao(2017)Ma and Hao]{ma2017multiple}
Ma, F. and Hao, J.-K.
\newblock A multiple search operator heuristic for the max-k-cut problem.
\newblock \emph{Annals of Operations Research}, 248:\penalty0 365--403, 2017.

\bibitem[Michael et~al.(2024)Michael, Bartels, Gonz{\'a}lez-Duque,
  Zainchkovskyy, Frellsen, Hauberg, and Boomsma]{michael2024continuous}
Michael, R., Bartels, S., Gonz{\'a}lez-Duque, M., Zainchkovskyy, Y., Frellsen,
  J., Hauberg, S., and Boomsma, W.
\newblock A continuous relaxation for discrete bayesian optimization.
\newblock \emph{arXiv preprint arXiv:2404.17452}, 2024.

\bibitem[Micheal(2002)]{trick2003}
Micheal, T.
\newblock The color datasets.
\newblock \url{https://mat.tepper.cmu.edu/COLOR02/}, 2002.

\bibitem[Morris et~al.(2019)Morris, Ritzert, Fey, Hamilton, Lenssen, Rattan,
  and Grohe]{morris2019weisfeiler}
Morris, C., Ritzert, M., Fey, M., Hamilton, W.~L., Lenssen, J.~E., Rattan, G.,
  and Grohe, M.
\newblock Weisfeiler and leman go neural: Higher-order graph neural networks.
\newblock \emph{Proceedings of the AAAI Conference on Artificial Intelligence},
  33\penalty0 (01):\penalty0 4602--4609, Jul. 2019.

\bibitem[Nath \& Kuhnle(2024)Nath and Kuhnle]{nath2024benchmark}
Nath, A. and Kuhnle, A.
\newblock A benchmark for maximum cut: Towards standardization of the
  evaluation of learned heuristics for combinatorial optimization.
\newblock \emph{arXiv preprint arXiv:2406.11897}, 2024.

\bibitem[Poland \& Zeugmann(2006)Poland and Zeugmann]{poland2006clustering}
Poland, J. and Zeugmann, T.
\newblock Clustering pairwise distances with missing data: Maximum cuts versus
  normalized cuts.
\newblock In \emph{International Conference on Discovery Science}, pp.\
  197--208. Springer, 2006.

\bibitem[Schuetz et~al.(2022)Schuetz, Brubaker, and
  Katzgraber]{schuetz2022combinatorial}
Schuetz, M.~J., Brubaker, J.~K., and Katzgraber, H.~G.
\newblock Combinatorial optimization with physics-inspired graph neural
  networks.
\newblock \emph{Nature Machine Intelligence}, 4\penalty0 (4):\penalty0
  367--377, 2022.

\bibitem[Shah et~al.(2024)Shah, Jain, Manchanda, Medya, and
  Ranu]{shah2024neurocut}
Shah, R., Jain, K., Manchanda, S., Medya, S., and Ranu, S.
\newblock Neurocut: A neural approach for robust graph partitioning.
\newblock In \emph{Proceedings of the 30th ACM SIGKDD Conference on Knowledge
  Discovery and Data Mining}, KDD '24, pp.\  2584–2595, New York, NY, USA,
  2024. Association for Computing Machinery.
\newblock ISBN 9798400704901.

\bibitem[Shinde et~al.(2021)Shinde, Narayanan, and
  Saunderson]{shinde2021memory}
Shinde, N., Narayanan, V., and Saunderson, J.
\newblock Memory-efficient approximation algorithms for max-k-cut and
  correlation clustering.
\newblock In \emph{Advances in Neural Information Processing Systems},
  volume~34, pp.\  8269--8281. Curran Associates, Inc., 2021.

\bibitem[T{\"o}nshoff et~al.(2021)T{\"o}nshoff, Ritzert, Wolf, and
  Grohe]{toenshoff2021graph}
T{\"o}nshoff, J., Ritzert, M., Wolf, H., and Grohe, M.
\newblock Graph neural networks for maximum constraint satisfaction.
\newblock \emph{Frontiers in artificial intelligence}, 3:\penalty0 580607,
  2021.

\bibitem[T{\"o}nshoff et~al.(2023)T{\"o}nshoff, Kisin, Lindner, and
  Grohe]{tonshoff2022one}
T{\"o}nshoff, J., Kisin, B., Lindner, J., and Grohe, M.
\newblock One model, any csp: Graph neural networks as fast global search
  heuristics for constraint satisfaction.
\newblock In \emph{Proceedings of the Thirty-Second International Joint
  Conference on Artificial Intelligence, {IJCAI-23}}, pp.\  4280--4288.
  International Joint Conferences on Artificial Intelligence Organization, 8
  2023.
\newblock Main Track.

\bibitem[Tsitsulin et~al.(2023)Tsitsulin, Palowitch, Perozzi, and
  M{\"u}ller]{tsitsulin2023graph}
Tsitsulin, A., Palowitch, J., Perozzi, B., and M{\"u}ller, E.
\newblock Graph clustering with graph neural networks.
\newblock \emph{Journal of Machine Learning Research}, 24\penalty0
  (127):\penalty0 1--21, 2023.

\bibitem[Ye(2003)]{ye2003}
Ye, Y.
\newblock The gset dataset.
\newblock \url{https://web.stanford.edu/~yyye/yyye/Gset/}, 2003.

\end{thebibliography}
\bibliographystyle{icml2025}

%%%%%%%%%%%%%%%%%%%%%%%%%%%%%%%%%%%%%%%%%%%%%%%%%%%%%%%%%%%%%%%%%%%%%%%%%%%%%%%
%%%%%%%%%%%%%%%%%%%%%%%%%%%%%%%%%%%%%%%%%%%%%%%%%%%%%%%%%%%%%%%%%%%%%%%%%%%%%%%
% APPENDIX
%%%%%%%%%%%%%%%%%%%%%%%%%%%%%%%%%%%%%%%%%%%%%%%%%%%%%%%%%%%%%%%%%%%%%%%%%%%%%%%
%%%%%%%%%%%%%%%%%%%%%%%%%%%%%%%%%%%%%%%%%%%%%%%%%%%%%%%%%%%%%%%%%%%%%%%%%%%%%%%
\newpage
\appendix
\onecolumn

\section{Related Works}

Relaxation-based methods have been central to the algorithmic design for Max-Cut and its generalizations. In Table~\ref{tab:relaxation}, we compare our proposed probability simplex relaxation with several representative approaches along key dimensions: variable complexity (\# Var.), applicability to general Max-$k$-Cut, polynomial-time solvability, objective value consistency with the original problem, and scalability to large instances.

\begin{table}[ht]
\centering
\caption{Comparison between Different Relaxations \label{tab:relaxation}}
\resizebox{\textwidth}{!}{%
\begin{tabular}{lccccc}
\toprule
\textbf{Relaxation} & \textbf{\# Var.} & \textbf{Max-$k$-Cut} & \textbf{Polynomial Solvable?} & \textbf{Obj. Value Consistency?} & \textbf{Scalable?} \\
\midrule
Lovasz Extension~\cite{bach2013learning} & $\mathcal{O}(N)$ & \XSolidBrush & \XSolidBrush & \Checkmark & \Checkmark \\
SDP Relaxation~\cite{goemans1995improved} & $\mathcal{O}(N^2)$ & \XSolidBrush & \Checkmark & \XSolidBrush & \XSolidBrush \\
SDP Relaxation~\cite{frieze1997improved} & $\mathcal{O}(N \times k)$ & \Checkmark & \Checkmark & \XSolidBrush & \XSolidBrush \\
Rank-2 Relaxation~\cite{burer2002rank} & $\mathcal{O}(N)$ & \XSolidBrush & \XSolidBrush & \XSolidBrush & \Checkmark \\
QUBO Relaxation~\cite{andrade2022physical} & $\mathcal{O}(N)$ & \XSolidBrush & \XSolidBrush & \XSolidBrush & \Checkmark \\
Probability Simplex Relaxation (ours) & $\mathcal{O}(N \times k)$ & \Checkmark & \XSolidBrush & \Checkmark & \Checkmark \\
\bottomrule
\end{tabular}
}
\end{table}

The Lovász extension~\cite{bach2013learning}, originally designed for submodular optimization, admits scalable convex formulations but does not extend naturally to general Max-$k$-Cut problems. Seminal SDP-based methods, such as Goemans-Williamson~\cite{goemans1995improved} for Max-Cut and its $k$-way extension~\cite{frieze1997improved}, offer polynomial-time approximation guarantees. However, their reliance on large-scale semidefinite programming limits practical scalability and makes them less effective on modern large-scale instances. Non-convex formulations, including the rank-2 relaxation~\cite{burer2002rank} and QUBO-based relaxation~\cite{andrade2022physical}, provide scalable alternatives for Max-Cut but lack theoretical guarantees for Max-$k$-Cut and are typically solved locally. These methods often exhibit poor objective consistency and limited generalization.

In contrast, our probability simplex relaxation introduces a non-convex yet tractable formulation for Max-$k$-Cut with $\mathcal{O}(N \times k)$ variables. While it is not globally solvable in polynomial time, its optimal value aligns exactly with that of the original Max-$k$-Cut problem. Empirically, our GNN-based solver produces high-quality fractional solutions, which serve as effective initializations for randomized sampling. Overall, the proposed relaxation strikes a favorable balance between expressiveness, consistency, and scalability, offering a practical and theoretically grounded solution framework for large-scale Max-$k$-Cut problems.

\section{Proof of Theorem \ref{theorem:local_optimal_neighborhood}}

\begin{proof}

Before proceeding with the proof of Theorem \ref{theorem:local_optimal_neighborhood}, we first define the neighborhood of a vector $\bar{\bm{x}} \in \Delta_k$, and establish results of Lemma \ref{lemma1} and Lemma \ref{lemma:prepare_local_optimal_neighborhood}.

\begin{definition} 
Let $\bar{\bm{x}}=(\bar{\bm{x}}_1,\cdots,\bar{\bm{x}}_k)$ denote a point in $\Delta_k$. We define the neighborhood induced by $\bar{\bm{x}}$ as follows:
\begin{equation*}
    \begin{aligned}
        \widetilde{\mathcal{N}}(\bar{\bm{x}})\coloneqq \left\{ (\bm{x}_1,\cdots,\bm{x}_k)\in\Delta_k \left|\sum_{j\in \mathcal{K}(\bar{\bm{x}})}\bm{x}_j=1 \right. \right\},
    \end{aligned}
\end{equation*}
where $\mathcal{K}(\bar{\bm{x}})=\left\{j\in\{1,\cdots,k\}\mid\bar{\bm{x}}_j>0\right\}$.
\end{definition}

\begin{lemma}
    \label{lemma1}
    Given $\bm{X}_{\cdot i}\in \widetilde{\mathcal{N}}(\bm{X}^{\star}_{\cdot i})$, it follows that 
    \begin{equation*}
        \begin{aligned}
            \mathcal{K}(\bm{X}_{\cdot i})\subseteq\mathcal{K}(\bm{X}^{\star}_{\cdot i}).
        \end{aligned}
    \end{equation*}
\end{lemma}
\begin{proof}
    Suppose there exists $j\in\mathcal{K}(\bm{X}_{\cdot i})$ such that $j\notin\mathcal{K}(\bm{X}^{\star}_{\cdot i})$, implying $\bm{X}_{ji}>0$ and $\bm{X}^{\star}_{ji}=0$.

    We then have
    \begin{equation*}
        \begin{aligned}
            \sum_{l\in\mathcal{K}(\bm{X}^{\star}_{\cdot i})} \bm{X}_{li}+\bm{X}_{ji}\le \sum_{l=1}^{k}\bm{X}_{li}=1,
        \end{aligned}
    \end{equation*}
    which leads to
    \begin{equation*}
        \begin{aligned}
            \sum_{l\in\mathcal{K}(\bm{X}^{\star}_{\cdot i})} \bm{X}_{li}\le 1-\bm{X}_{ji}<1,
        \end{aligned}
    \end{equation*}
    contradicting with the fact that $\bm{X}_{\cdot i} \in \widetilde{\mathcal{N}}(\bm{X}^{\star}_{\cdot i})$.
\end{proof}

\begin{lemma}
    Let $\bm{X}^{\star}$ be a globally optimal solution to $\overline{\textbf{P}}$, then
    \begin{equation*}
        f(\bm{X};\bm{W})=f(\bm{X}^{\star};\bm{W}),
    \end{equation*}
    where $\bm{X}$ has only the $i^{th}$ column $\bm{X}_{\cdot i}\in\widetilde{\mathcal{N}}(\bm{X}^{\star}_{\cdot i})$, and other columns are identical to those of $\bm{X}^{\star}$. Moreover, $\bm{X}$ is also a globally optimal solution to $\bar{\bm{P}}$. 
    %where $\bm{X}=[\bar{\bm{x}}_1,\cdots,\bar{\bm{x}}_{i-1},\bm{x}_i,\bar{\bm{x}}_{i+1},\cdots,\bar{\bm{x}}_N]$ for $i\in\mathcal{V}$, and $\bm{x}_i\in\tilde{\mathcal{N}}(\bar{\bm{x}}_i)$. 
    \label{lemma:prepare_local_optimal_neighborhood}
\end{lemma}
\begin{proof}

    The fact that $\bm{X}$ is a globally optimal solution to $\bar{\bm{P}}$ follows directly from the equality $f(\bm{X};\bm{W})=f(\bm{X}^{\star};\bm{W})$. Thus, it suffices to prove this equality. 
    Consider that $\bm{X}^{\star}$ and $\bm{X}$ differ only in the $i^{th}$ column, and $\bm{X}_{\cdot i}\in \widetilde{\mathcal{N}}(\bm{X}^{\star}_{\cdot i})$. We can rewrite the objective value function as
    \begin{equation*}
        \begin{aligned}
            f(\bm{X};\bm{W}) = g(\bm{X}_{\cdot i};\bm{X}_{\cdot -i}) + h(\bm{X}_{\cdot -i}),
        \end{aligned}
    \end{equation*}
    where $\bm{X}_{\cdot -i}$ represents all column vectors of $\bm{X}$ except the $i^{th}$ column. The functions $g$ and $h$ are defined as follows:
    \begin{equation*}
        \begin{aligned}
            g(\bm{X}_{\cdot i};\bm{X}_{\cdot -i})&=\sum_{j=1}^N\bm{W}_{ij}\bm{X}_{\cdot i}^\top\bm{X}_{\cdot j}+\sum_{j=1}^N\bm{W}_{ji}\bm{X}_{\cdot j}^\top\bm{X}_{\cdot i} - \bm{W}_{ii}\bm{X}_{\cdot i}^\top\bm{X}_{\cdot i}, \\
            h(\bm{X}_{\cdot -i})&=\sum_{l=1,l\neq i}^N\sum_{j=1,j\neq i}^N\bm{W}_{lj}\bm{X}_{\cdot l}^\top\bm{X}_{\cdot j}
        \end{aligned}
    \end{equation*}
    To establish that $f(\bm{X};\bm{W})=f(\bm{X}^{\star};\bm{W})$, it suffices to show that
    \begin{equation*}
        g(\bm{X}_{\cdot i};\bm{X}_{\cdot -i}) = g(\bm{X}^{\star}_{\cdot i};\bm{X}_{\cdot -i})
    \end{equation*}
    as $\bm{X}_{\cdot -i}=\bm{X}^{\star}_{\cdot -i}$.

    Rewriting $g(\bm{X}_{\cdot i};\bm{X}_{\cdot -i})$, we obtain
    \begin{equation*}
        \begin{aligned}
            g(\bm{X}_{\cdot i};\bm{X}_{\cdot -i})&=\sum_{j=1}^N \bm{W}_{ij}\bm{X}_{\cdot i}^\top\bm{X}_{\cdot j} + \sum_{j=1}^N \bm{W}_{ji}\bm{X}_{\cdot j}^\top\bm{X}_{\cdot i}  \\
            & = 2 \sum_{j=1}^N\bm{W}_{ij}\bm{X}_{\cdot i}^\top \bm{X}_{\cdot j} \\
            & = 2 \bm{X}_{\cdot i}^\top\sum_{j=1,j\neq i}^{N} \bm{W}_{ij} \bm{X}_{\cdot j} \\
            & = 2 \bm{X}_{\cdot i}^\top\bm{Y}_{\cdot i},
        \end{aligned}
    \end{equation*}
    where $\bm{Y}_{\cdot i}:=\sum_{j=1, j\neq i}^{N}\bm{W}_{ij}\bm{X}_{\cdot j}$.

    If $\vert\mathcal{K}(\bm{X}^{\star}_{\cdot i})\vert=1$, then there is only one non-zero element in $\bm{X}^{\star}_{\cdot i}$ equal to one. Therefore, $g(\bm{X}^{\star}_{\cdot i};\bm{X}_{\cdot -i})=g(\bm{X}_{\cdot i};\bm{X}_{\cdot -i})$ since $\bm{X}_{\cdot i}=\bm{X}^{\star}_{\cdot i}$.

    For the case where $\vert\mathcal{K}(\bm{X}^{\star}_{\cdot i})\vert>1$, we consider any indices $j, l\in\mathcal{K}(\bm{X}^{\star}_{\cdot i})$ such that $\bm{X}^{\star}_{ji}, \bm{X}^{\star}_{li}\in(0, 1)$. Then, there exists $\epsilon>0$ such that we can construct a point $\widetilde{\bm{x}}\in\Delta_k$ where the $j^{th}$ element is set to $\bm{X}^{\star}_{ji}-\epsilon$, the $l^{th}$ element is set to $\bm{X}^{\star}_{li}+\epsilon$, and all other elements remain the same as in $\bm{X}^{\star}_{\cdot i}$. Since $\bm{X}^{\star}$ is a globally optimum of the function $f(\bm{X};\bm{W})$, it follows that $\bm{X}^{\star}_{\cdot i}$ is also a global optimum for the function $g(\bm{X}^{\star}_{\cdot i};\bm{X}_{\cdot -i})$. Thus, we have
    \begin{equation*}
        \begin{aligned}
            g(\bm{X}^{\star}_{\cdot i};\bm{X}_{\cdot -i})&\le g(\widetilde{\bm{x}};\bm{X}_{\cdot -i}) \notag \\
            \bm{X}^{\star \top}_{\cdot i}\bm{Y}_{\cdot i}&\le \widetilde{\bm{x}}^\top\bm{Y}_{\cdot i} \notag \\
            &=\bm{X}^{\star\top}_{\cdot i}\bm{Y}_{\cdot i} - \epsilon \bm{Y}_{ji} + \epsilon \bm{Y}_{li},
        \end{aligned}
    \end{equation*}
    which leads to the inequality
    \begin{equation}
        \begin{aligned}
            \bm{Y}_{ji} \le \bm{Y}_{li}. \label{lemma1:ineqn:1}
        \end{aligned}
    \end{equation}

    Next, we can similarly construct another point $\hat{\bm{x}}\in\Delta_k$ with its $j^{th}$ element equal to $\bm{X}^{\star}_{ji}+\epsilon$, the $k^{th}$ element equal to $\bm{X}^{\star}_{ki}-\epsilon$, and all other elements remain the same as in $\bm{X}^{\star}_{\cdot i}$. Subsequently, we can also derive that
    \begin{equation*}
        \begin{aligned}
            g(\bm{X}^{\star}_{\cdot i};\bm{X}_{\cdot -i})&\le g(\hat{\bm{x}};\bm{X}_{\cdot -i}) \notag \\
            &=\bm{X}^{\star\top}_{\cdot i}\bm{Y}_{\cdot i}+\epsilon \bm{Y}_{ji}-\epsilon\bm{Y}_{li},
        \end{aligned}
    \end{equation*}
    which leads to another inequality
    \begin{equation}
        \begin{aligned}
            \bm{Y}_{li}\le \bm{Y}_{ji}. \label{lemma1:ineqn:2}
        \end{aligned}
    \end{equation}

    Consequently, combined inequalities (\ref{lemma1:ineqn:1}) and (\ref{lemma1:ineqn:2}), we have
    \begin{equation*}
        \bm{Y}_{ji}=\bm{Y}_{li},
    \end{equation*}
    for $j,l\in \mathcal{K}(\bm{X}^{\star}_{\cdot i})$.

    From this, we can deduce that
    \begin{equation*}
        \begin{aligned}
            \bm{Y}_{j_1i}=\bm{Y}_{j_2i}=\cdots =\bm{Y}_{j_{\vert\mathcal{K}(\bm{X}^{\star}_{\cdot i})\vert}i}= t,
        \end{aligned}
    \end{equation*}
    where $j_1,\cdots,j_{\vert\mathcal{K}(\bm{X}^{\star}_{\cdot i})\vert}\in\mathcal{K}(\bm{X}^{\star}_{\cdot i})$.

    Next, we find that
    \begin{equation*}
        \begin{aligned}
        g(\bm{X}^{\star}_{\cdot i};\bm{X}_{\cdot -i})&=2\bm{X}^{\star\top}_{\cdot i}\bm{Y}_{\cdot i} \\
        &=2\sum_{j=1}^{k}\bm{X}^{\star}_{ji}\bm{Y}_{ji}    \\
        &=2\sum_{j=1, j\in\mathcal{K}(\bm{X}^{\star}_{\cdot i})}^{N} \bm{X}^{\star}_{ji}\bm{Y}_{ji} \\
        &=2t\sum_{j=1, j\in\mathcal{K}(\bm{X}^{\star}_{\cdot i})}^{N} \bm{X}^{\star}_{ji} \\
        &=2t.
        \end{aligned}
    \end{equation*}

    Similarly, we have
    \begin{equation*}
        \begin{aligned}
            g(\bm{X}_{\cdot i};\bm{X}_{\cdot -i})&=2\bm{X}_{\cdot i}^\top\bm{Y}_{\cdot i} \\
            &=2\sum_{j=1}^k\bm{X}_{ji}\bm{Y}_{ji} \\
            &=2\sum_{j=1,j\in\mathcal{K}(\bm{X}_{\cdot i})} \bm{X}_{ji}\bm{Y}_{ji} \\
            &\overset{\text{Lemma \ref{lemma1}}}{=}2t\sum_{j=1,j\in\mathcal{K}(\bm{X}_{\cdot i})} \bm{X}_{ji}\\
            &=2t
        \end{aligned}
    \end{equation*}

    Accordingly, we conclude that
    \begin{equation*}
        \begin{aligned}
            g(\bm{X}_{\cdot i};\bm{X}_{\cdot -i})=g(\bm{X}^{\star}_{\cdot i};\bm{X}_{\cdot -i}),
        \end{aligned}
    \end{equation*}
    which leads us to the result
    \begin{equation*}
        \begin{aligned}
            f(\bm{X};\bm{W})=f(\bm{X}^{\star};\bm{W}),
        \end{aligned}
    \end{equation*}
    where $\bm{X}_{\cdot i}\in \widetilde{\mathcal{N}}(\bm{X}^{\star}_{\cdot i})$, $\bm{X}_{\cdot -i}=\bm{X}^{\star}_{\cdot -i}$.
\end{proof}

Accordingly, for any \(\bm{X} \in \mathcal{N}(\bm{X}^{\star})\), we iteratively apply Lemma \ref{lemma:prepare_local_optimal_neighborhood} to each column of $\bm{X}^{\star}$ while holding the other columns fixed, thereby proving Theorem \ref{theorem:local_optimal_neighborhood}.

\end{proof}

\section{Proof of Theorem \ref{theorem:equal_obj_value}}
\begin{proof}

Based on $\overline{\bm{X}}$, we can construct the random variable $\widetilde{\bm{X}}$, where $\widetilde{\bm{X}}_{\cdot i} \sim \text{Cat}(\bm{x};\bm{p}=\overline{\bm{X}}_{\cdot i})$. The probability mass function is given by
\begin{equation}
\begin{aligned}
    \textbf{P}(\widetilde{\bm{X}}_{\cdot i}=\bm{e}_\ell)=\overline{\bm{X}}_{\ell i}, \label{for:construct}
\end{aligned} 
\end{equation}
where $\ell=1,\cdots,k$. 

Next, we have
    \begin{align}
        \mathbb{E}_{\widetilde{\bm{X}}}[f(\widetilde{\bm{X}};\bm{W})]&=\mathbb{E}_{\widetilde{\bm{X}}}[\widetilde{\bm{X}}\bm{W}\widetilde{\bm{X}}^\top]=\mathbb{E}_{\widetilde{\bm{X}}}[\sum_{i=1}^N\sum_{j=1}^N \bm{W}_{ij} \widetilde{\bm{X}}_{\cdot i}^\top\widetilde{\bm{X}}_{\cdot j}] \notag\\
        &=\sum_{i=1}^N\sum_{j=1}^N\bm{W}_{ij}\mathbb{E}_{\widetilde{\bm{X}}_{\cdot i}\widetilde{\bm{X}}_{\cdot j}}[\widetilde{\bm{X}}_{\cdot i}^\top\widetilde{\bm{X}}_{\cdot j}] \notag\\
        &=\sum_{i=1}^N\sum_{j=1}^N\bm{W}_{ij}\mathbb{E}_{\widetilde{\bm{X}}_{\cdot i}\widetilde{\bm{X}}_{\cdot j}}[\mathbbm{1}(\widetilde{\bm{X}}_{\cdot i}=\widetilde{\bm{X}}_{\cdot j})] \notag \\
        &=\sum_{i=1}^N\sum_{j=1}^N\bm{W}_{ij}\mathbb{P}(\widetilde{\bm{X}}_{\cdot i}=\widetilde{\bm{X}}_{\cdot j}) \notag\\
        &=\sum_{i=1}^N\sum_{j=1,j\neq i}^N\bm{W}_{ij}\mathbb{P}(\widetilde{\bm{X}}_{\cdot i}=\widetilde{\bm{X}}_{\cdot j}).  \label{for:7}
    \end{align}

Since $\widetilde{\bm{X}}_{\cdot i}$ and $\widetilde{\bm{X}}_{\cdot j}$ are independent for $i\neq j$, we have
\begin{align}
    \mathbb{P}(\widetilde{\bm{X}}_{\cdot i}=\widetilde{\bm{X}}_{\cdot j})&=\sum_{\ell=1}^k\mathbb{P}(\widetilde{\bm{X}}_{\cdot i}=\widetilde{\bm{X}}_{\cdot j}=\bm{e}_\ell) \notag\\
        &=\sum_{\ell=1}^k\mathbb{P}(\widetilde{\bm{X}}_{\cdot i}=\bm{e}_\ell,\widetilde{\bm{X}}_{\cdot j}=\bm{e}_\ell) \notag \\
        &=\sum_{\ell=1}^k\mathbb{P}(\widetilde{\bm{X}}_{\cdot i}=\bm{e}_\ell)\mathbb{P}(\widetilde{\bm{X}}_{\cdot j}=\bm{e}_\ell) \notag \\
        &=\sum_{\ell=1}^k\overline{\bm{X}}_{\ell i}\overline{\bm{X}}_{\ell j} \notag \\
        &=\overline{\bm{X}}_{\cdot i}^\top\overline{\bm{X}}_{\cdot j}.\label{for:8}
\end{align}

Substitute (\ref{for:8}) into (\ref{for:7}), we obtain
\begin{align}
    \mathbb{E}_{\widetilde{\bm{X}}}[f(\widetilde{\bm{X}};\bm{W})]=\sum_{i=1}^N\sum_{j=1}^N\bm{W}_{ij}\overline{\bm{X}}_{\cdot i}^\top\overline{\bm{X}}_{\cdot j}=f(\overline{\bm{X}};\bm{W}). \label{res:thm1}
\end{align}

\end{proof}

\section{The Results on Unweighted Gset Instances \label{app:gset}}

\begin{sidewaystable}
  \centering
  \caption{Complete results on Gset instances for Max-Cut.}
  \label{tab:3}
\resizebox{\textwidth}{!}{%
\begin{tabular}{ccccccccccccccccccccccccc}
\hline
\multirow{2}{*}{Instance} & \multirow{2}{*}{$\vert\mathcal{V}\vert$} & \multirow{2}{*}{$\vert\mathcal{E}\vert$} & \multicolumn{2}{c}{\texttt{GW}} & \multicolumn{2}{c}{\texttt{MD}} & \multicolumn{2}{c}{\texttt{PI-GNN}} & \multicolumn{2}{c}{\texttt{Genetic}} & \multicolumn{2}{c}{\texttt{BQP}} & \multicolumn{2}{c}{\texttt{ECO-DQN}} & \multicolumn{2}{c}{\texttt{ANYCSP}} & \multicolumn{2}{c}{\texttt{MOH}} & \multicolumn{2}{c}{\texttt{LPI}} & \multicolumn{2}{c}{\texttt{ROS-vanilla}} & \multicolumn{2}{c}{\texttt{ROS}} 
 \\ \cmidrule(l){4-5}  \cmidrule(l){6-7} \cmidrule(l){8-9} \cmidrule(l){10-11} \cmidrule(l){12-13} \cmidrule(l){14-15} \cmidrule(l){16-17} \cmidrule(l){18-19} \cmidrule(l){20-21} \cmidrule(l){22-23} \cmidrule(l){24-25} 
                          &                                          &                                          & Obj. $\uparrow$    & Time (s) $\downarrow$   & Obj. $\uparrow$    & Time (s) $\downarrow$   & Obj. $\uparrow$     & Time (s) $\downarrow$     & Obj. $\uparrow$      & Time (s) $\downarrow$     & Obj. $\uparrow$      & Time (s) $\downarrow$      & Obj. $\uparrow$    & Time (s) $\downarrow$    & Obj. $\uparrow$      & Time (s) $\downarrow$      & Obj. $\uparrow$    & Time (s) $\downarrow$    & Obj. $\uparrow$    & Time (s) $\downarrow$    & Obj. $\uparrow$        & Time (s) $\downarrow$        & Obj. $\uparrow$    & Time (s) $\downarrow$    \\
\hline
\end{tabular}%
}
\end{sidewaystable}

\begin{sidewaystable}
  \centering
  \caption*{Table 6: Continued.}
 \resizebox{\textwidth}{!}{%
\begin{tabular}{ccccccccccccccccccccccccc}
\hline
\multirow{2}{*}{Instance} & \multirow{2}{*}{$\vert\mathcal{V}\vert$} & \multirow{2}{*}{$\vert\mathcal{E}\vert$} & \multicolumn{2}{c}{\texttt{GW}} & \multicolumn{2}{c}{\texttt{MD}} & \multicolumn{2}{c}{\texttt{PI-GNN}} & \multicolumn{2}{c}{\texttt{Genetic}} & \multicolumn{2}{c}{\texttt{BQP}} & \multicolumn{2}{c}{\texttt{ECO-DQN}} & \multicolumn{2}{c}{\texttt{ANYCSP}} & \multicolumn{2}{c}{\texttt{MOH}} & \multicolumn{2}{c}{\texttt{LPI}} & \multicolumn{2}{c}{\texttt{ROS-vanilla}} & \multicolumn{2}{c}{\texttt{ROS}}  
\\ \cmidrule(l){4-5}  \cmidrule(l){6-7} \cmidrule(l){8-9} \cmidrule(l){10-11} \cmidrule(l){12-13} \cmidrule(l){14-15} \cmidrule(l){16-17} \cmidrule(l){18-19} \cmidrule(l){20-21} \cmidrule(l){22-23} \cmidrule(l){24-25} 
                          &                                          &                                          & Obj. $\uparrow$    & Time (s) $\downarrow$   & Obj. $\uparrow$    & Time (s) $\downarrow$   & Obj. $\uparrow$     & Time (s) $\downarrow$     & Obj. $\uparrow$      & Time (s) $\downarrow$     & Obj. $\uparrow$      & Time (s) $\downarrow$      & Obj. $\uparrow$    & Time (s) $\downarrow$    & Obj. $\uparrow$      & Time (s) $\downarrow$      & Obj. $\uparrow$    & Time (s) $\downarrow$    & Obj. $\uparrow$    & Time (s) $\downarrow$    & Obj. $\uparrow$        & Time (s) $\downarrow$        & Obj. $\uparrow$    & Time (s) $\downarrow$    \\
\hline
\end{tabular}%
}
\end{sidewaystable}

\begin{sidewaystable}
  \centering
\tiny  
  \caption{Complete results on Gset instances for Max-$3$-Cut.}
  \label{tab:4}
   \resizebox{\textwidth}{!}{%
\begin{tabular}{ccccccccccccccccc}
\midrule
\multirow{2}{*}{Instance} & \multirow{2}{*}{$\vert\mathcal{V}\vert$} & \multirow{2}{*}{$\vert\mathcal{E}\vert$} & \multicolumn{2}{c}{\texttt{MD}}        & \multicolumn{2}{c}{\texttt{Genetic}}                           & \multicolumn{2}{c}{\texttt{BQP}}      & \multicolumn{2}{c}{\texttt{ANYCSP}}                          & \multicolumn{2}{c}{\texttt{MOH}}       & \multicolumn{2}{c}{\texttt{ROS-vanilla}} & \multicolumn{2}{c}{\texttt{ROS}}       \\
\cmidrule(l){4-5}  \cmidrule(l){6-7} \cmidrule(l){8-9} \cmidrule(l){10-11} \cmidrule(l){12-13} \cmidrule(l){14-15} \cmidrule(l){16-17}
                          &                      &                      & Obj. $\uparrow$ & Time (s) $\downarrow$ & Obj. $\uparrow$        & Time (s) $\downarrow$                  & Obj. $\uparrow$        & Time (s) $\downarrow$                  & Obj. $\uparrow$ & Time (s) $\downarrow$ & Obj. $\uparrow$  & Time (s) $\downarrow$  & Obj. $\uparrow$ & Time (s) $\downarrow$ & Obj. $\uparrow$ & Time (s) $\downarrow$ \\ \midrule
G1                        & 800                                      & 19176                                    & 14735              & 9.6                     & 14075                & 595.3                      & 14880              & 16.5                     & 15115                & 180.1                     & 15165              & 557.3                    & 14949                  & 2.8                          & 14961              & 1.9                      \\
G2                        & 800                                      & 19176                                    & 14787              & 8.4                     & 14035                & 595.3                      & 14845              & 17.0                     & 15088                & 180.1                     & 15172              & 333.3                    & 15033                  & 2.8                          & 14932              & 2.3                      \\
\hline
\end{tabular}%
}
\end{sidewaystable}

\begin{sidewaystable}
\centering
\tiny
\caption*{Table 7: Continued.}
  \resizebox{\textwidth}{!}{%
\begin{tabular}{ccccccccccccccccc}
\midrule
\multirow{2}{*}{Instance} & \multirow{2}{*}{$\vert\mathcal{V}\vert$} & \multirow{2}{*}{$\vert\mathcal{E}\vert$} & \multicolumn{2}{c}{\texttt{MD}}        & \multicolumn{2}{c}{\texttt{Genetic}}                           & \multicolumn{2}{c}{\texttt{BQP}}       & \multicolumn{2}{c}{\texttt{ANYCSP}}                        & \multicolumn{2}{c}{\texttt{MOH}}       & \multicolumn{2}{c}{\texttt{ROS-vanilla}} & \multicolumn{2}{c}{\texttt{ROS}}       \\
\cmidrule(l){4-5}  \cmidrule(l){6-7} \cmidrule(l){8-9} \cmidrule(l){10-11} \cmidrule(l){12-13} \cmidrule(l){14-15} \cmidrule(l){16-17} 
                          &                      &                      & Obj. $\uparrow$ & Time (s) $\downarrow$ & Obj. $\uparrow$        & Time (s) $\downarrow$                  & Obj. $\uparrow$        & Time (s) $\downarrow$                  & Obj. $\uparrow$ & Time (s) $\downarrow$ & Obj. $\uparrow$  & Time (s) $\downarrow$  & Obj. $\uparrow$ & Time (s) $\downarrow$ \\ \midrule
G36                       & 2000                                     & 11766                                    & 9600               & 13.6                    & N/A                  & N/A                        & 9786               & 138.6                    & 9972                 & 180.1                     & 10039              & 74.3                     & 9372                   & 2.1                          & 9581               & 2.3                      \\
G37                       & 2000                                     & 11785                                    & 9632               & 14.9                    & N/A                  & N/A                        & 9821               & 139.2                    & 9983                 & 180.1                     & 10052              & 3.4                      & 8893                   & 1.4                          & 9422               & 1.5                      \\
\hline
\end{tabular}%
}
\end{sidewaystable}

\newpage

\section{The Results on Weighted Gset Instances \label{app:full_weighted_gset}}

The computational time on weighted Gset with perturbation ratio of $[0.9, 1.1]$ and $[0, 100]$ are shown in Fig.~\ref{fig:full_time1} and Fig.~\ref{fig:full_time2} respectively. The cut values are shown in Table~\ref{tab:weighted_gset2} and Table~\ref{tab:weighted_gset3} respectively.

\begin{figure}[t]
    \centering
    \begin{subfigure}{0.49\textwidth}
        \centering      \includegraphics[width=\textwidth]{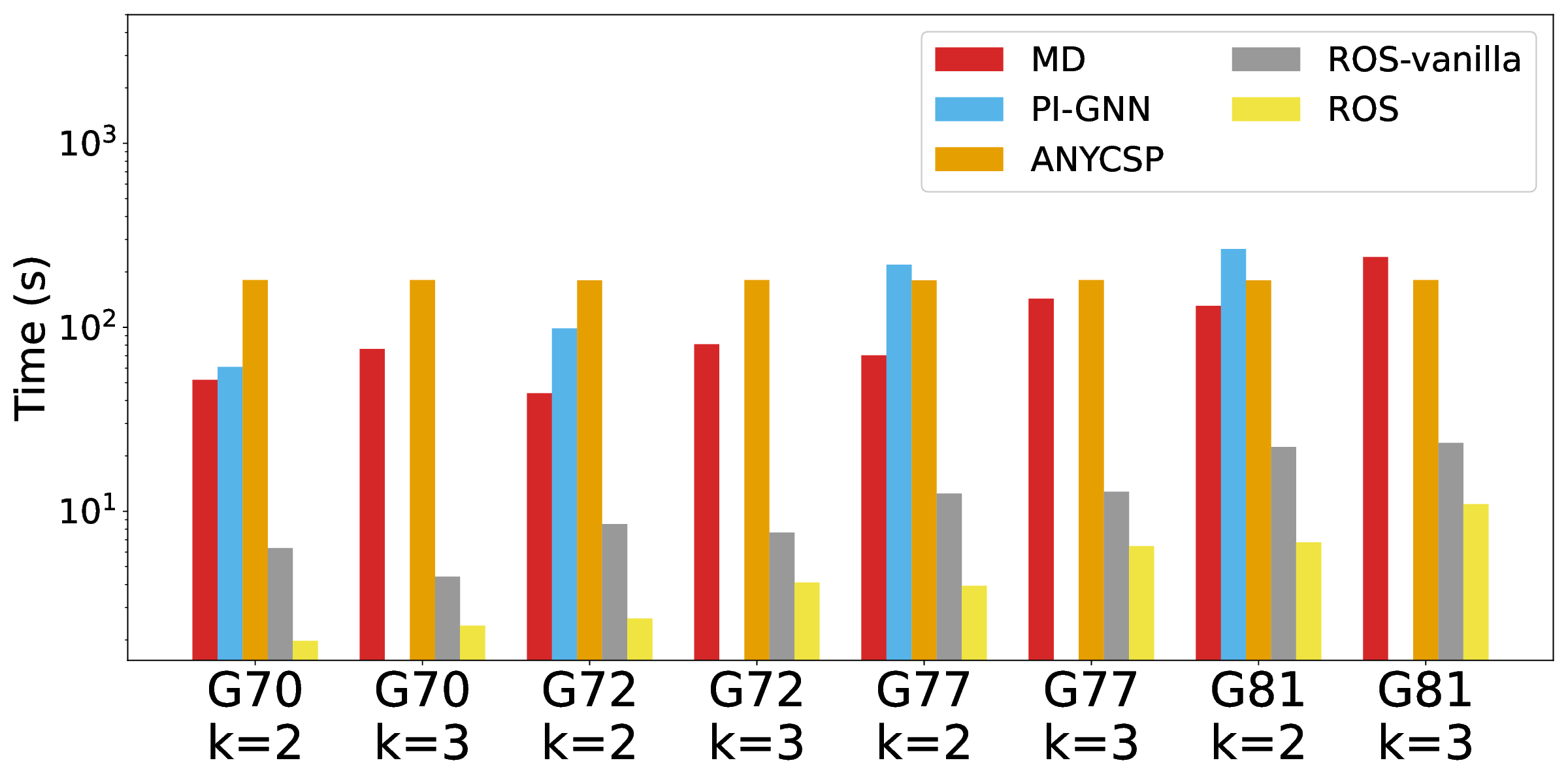}
        \caption{Weighted Gset with perturbation ratio $[0.9, 1.1]$\label{fig:full_time1}}
    \end{subfigure}
    \hfill
    \begin{subfigure}{0.49\textwidth}
        \centering
        \includegraphics[width=\textwidth]{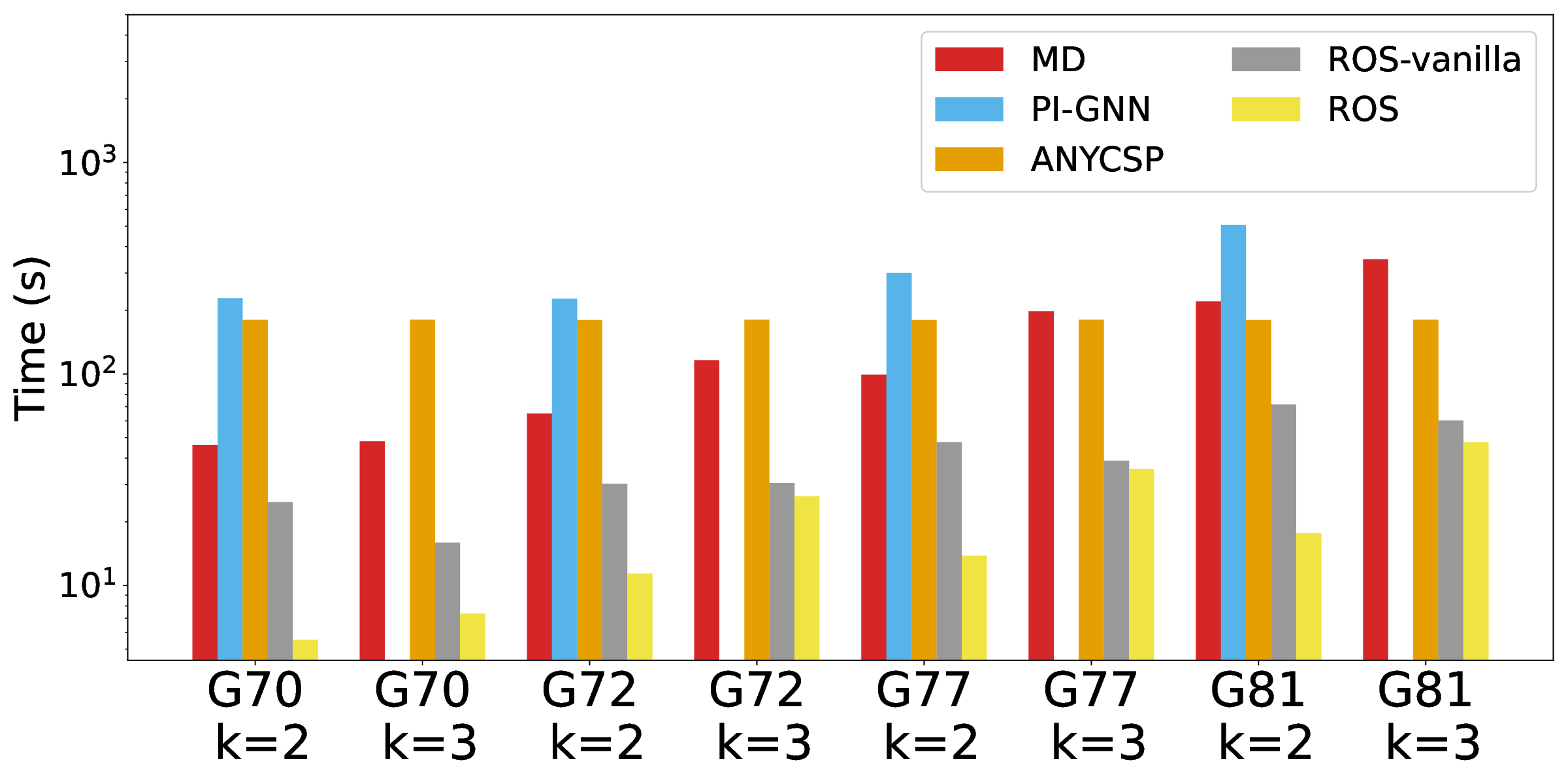}
        \caption{Weighted Gset with perturbation ratio $[0, 100]$\label{fig:full_time2}}
    \end{subfigure}
    \caption{The computational time comparison of Max-$k$-Cut problems.
    \label{fig:full_time}}
\end{figure}

\begin{table}[t]
\centering
\caption{Cut value comparison of Max-$k$-Cut problems on weighted Gset instances with perturbation ratio $[0.9, 1.1]$.}
\label{tab:weighted_gset2}
\resizebox{\textwidth}{!}{%
\begin{tabular}{ccccccccc}
\midrule
\multirow{2}{*}{Methods}                      & \multicolumn{2}{c}{\texttt{G70 (N=10,000)}}              & \multicolumn{2}{c}{\texttt{G72 (N=10,000)}}                 & \multicolumn{2}{c}{\texttt{G77 (N=14,000)}}                 & \multicolumn{2}{c}{\texttt{G81 (N=20,000)}}                \\\cmidrule(l){2-3}\cmidrule(l){4-5}\cmidrule(l){6-7}\cmidrule(l){8-9}
                           & $k=2$   &$k=3$&$k=2$   &$k=3$& $k=2$   &$k=3$& $k=2$ &$k=3$   \\ \midrule
\texttt{MD}&$8594.38$&$9709.56$&$5647.28$&$6585.04$&$8051.81$&$9337.31$&$11326.30$&$13179.33$                           \\ 
\texttt{PI-GNN}&$8422.79$&\textbf{--}&$5309.65$&\textbf{--}&$7470.89$&\textbf{--}&$10416.44$&\textbf{--}                            \\ 
\texttt{ANYCSP}  & $5198.87$&$5375.92$& $-15.57$&$-25.33$& $81.76$& $114.36$ & $33.49$& $-4.25$ \\ \midrule
\texttt{ROS-vanilla}&$9177.21$ & $9991.95$ &  $6542.78$ & $7733.87$ & $9265.65$ &  $10944.35$ &$13132.52$& $15456.28$\\
\texttt{ROS}     & $8941.80$&$9970.72$& $6165.62$&$7366.54$& $8737.59$&$10359.25$&    $12325.85$&$14570.04$  \\\midrule
\end{tabular}%
}
\end{table}

\begin{table}[t]
\centering
\caption{Cut value comparison of Max-$k$-Cut problems on weighted Gset instances with perturbation ratio $[0, 100]$.}
\label{tab:weighted_gset3}
\resizebox{\textwidth}{!}{%
\begin{tabular}{ccccccccc}
\midrule
\multirow{2}{*}{Methods}                      & \multicolumn{2}{c}{\texttt{G70 (N=10,000)}}              & \multicolumn{2}{c}{\texttt{G72 (N=10,000)}}                 & \multicolumn{2}{c}{\texttt{G77 (N=14,000)}}                 & \multicolumn{2}{c}{\texttt{G81 (N=20,000)}}                \\\cmidrule(l){2-3}\cmidrule(l){4-5}\cmidrule(l){6-7}\cmidrule(l){8-9}
                           & $k=2$   &$k=3$&$k=2$   &$k=3$& $k=2$   &$k=3$& $k=2$ &$k=3$   \\ \midrule
\texttt{MD}&$456581.30$&$497167.48$&$338533.07$&$392908.80$&$482413.19$&$558264.21$&$682809.47$&$790089.41$                           \\ 
\texttt{PI-GNN}&$442650.59$&\textbf{--}&$312802.48$&\textbf{--}&$442354.44$&\textbf{--}&$623256.74$&\textbf{--}                            \\ 
\texttt{ANYCSP}  & $467696.98$&$491654.75$& $-1903.50$&$-2498.86$& $9712.13$& $10130.89$ & $ 2842.64$& $2845.46$ \\ \midrule
\texttt{ROS-vanilla}&$472067.14$ & $498273.60$ &  $367795.62$ & $421189.97$ & $524010.92$ &  $597597.50$ &$742432.41$& $846395.85$\\
\texttt{ROS}     & $470268.97$&$498269.90$& $362910.89$&$415905.88$& $515991.31$&$590312.40$&    $731468.67$&$835424.19$  \\\midrule
\end{tabular}%
}
\end{table}

\section{Ablation Study}

\subsection{Model Ablation}
We conducted additional ablation studies to clarify the contributions of different modules.

\textbf{Effect of Neural Networks:} We consider two cases: (i) replace GNNs by multi-layer perceptrons (denoted by \texttt{ROS-MLP}) in our ROS framework and (ii) solve the relaxation via mirror descent (denoted by \texttt{MD}). Experiments on the Gset dataset show that \texttt{ROS} consistently outperforms \texttt{ROS-MLP} and \texttt{MD}, highlighting the benefits of using GNNs for the relaxation step.

\textbf{Effect of Random Sampling:} We compared \texttt{ROS} with \texttt{PI-GNN}, which employs heuristic rounding instead of our random sampling algorithm. Results indicate that \texttt{ROS} generally outperforms \texttt{PI-GNN}, demonstrating the importance of the sampling procedure.  

These comparisons, detailed in Tables \ref{tab:gset_objective_ablataion} and \ref{tab:gset_time_ablataion}, confirm that both the GNN-based optimization and the random sampling algorithm contribute significantly to the overall performance.

\begin{table}[t]
\centering
\caption{Cut values returned by each method on Gset.}
\begin{tabular}{ccccccccc}
\toprule
\multirow{2}{*}{Methods} & \multicolumn{2}{c}{\texttt{G70}} & \multicolumn{2}{c}{\texttt{G72}} & \multicolumn{2}{c}{\texttt{G77}} & \multicolumn{2}{c}{\texttt{G81}} \\
 \cmidrule(l){2-3}  \cmidrule(l){4-5}  \cmidrule(l){6-7} \cmidrule(l){8-9}
                                  & $k=2$      & $k=3$      & $k=2$      & $k=3$      & $k=2$      & $k=3$      & $k=2$      & $k=3$      \\
\midrule
\texttt{ROS-MLP}                & $8867$     & $9943$     & $6052$     & $6854$     & $8287$     & $9302$     & $12238$    & $12298$    \\
\texttt{PI-GNN}                  & $8956$     & --         & $4544$     & --         & $6406$     & --         & $8970$     & --         \\
\texttt{MD}                      & $8551$     & $9728$     & $5638$     & $6612$     & $7934$     & $9294$     & $11226$    & $13098$    \\
\texttt{ROS}                     & $8916$     & $9971$     & $6102$     & $7297$     & $8740$     & $10329$    & $12332$    & $14464$    \\
\bottomrule
\end{tabular}
\label{tab:gset_objective_ablataion}
\end{table}

\begin{table}[t]
\centering
\caption{Computational time for each method on Gset.}
\begin{tabular}{ccccccccc}
\toprule
\multirow{2}{*}{Methods} & \multicolumn{2}{c}{\texttt{G70}} & \multicolumn{2}{c}{\texttt{G72}} & \multicolumn{2}{c}{\texttt{G77}} & \multicolumn{2}{c}{\texttt{G81}} \\
 \cmidrule(l){2-3}  \cmidrule(l){4-5}  \cmidrule(l){6-7} \cmidrule(l){8-9}
                                  & $k=2$      & $k=3$      & $k=2$      & $k=3$      & $k=2$      & $k=3$      & $k=2$      & $k=3$      \\
\midrule
\texttt{ROS-MLP}                & $3.49$     & $3.71$     & $3.93$     & $4.06$     & $8.39$     & $9.29$     & $11.98$    & $16.97$    \\
\texttt{PI-GNN}                  & $34.50$    & --         & $253.00$   & --         & $349.40$   & --         & $557.70$   & --         \\
\texttt{MD}                      & $54.30$    & $74.80$    & $44.20$    & $79.20$    & $66.00$    & $142.30$   & $130.80$   & $241.10$   \\
\texttt{ROS}                     & $3.40$     & $2.50$     & $3.90$     & $3.50$     & $8.10$     & $8.50$     & $9.30$     & $9.70$     \\
\bottomrule
\end{tabular}
\label{tab:gset_time_ablataion}
\end{table}

\subsection{Sample Effect Ablation}

We investigated the effect of the number of sampling iterations and report the results in Tables  \ref{tab:gset_objective_T}, \ref{tab:gset_sampling_time_T}, \ref{tab:regular_objective_T}, and \ref{tab:regular_sampling_time_T}.

\textbf{Cut Value} (Table \ref{tab:gset_objective_T}, Table \ref{tab:regular_objective_T}): The cut values stabilize after approximately 5 sampling iterations, demonstrating strong performance without requiring extensive sampling. 

\textbf{Sampling Time} (Table \ref{tab:gset_sampling_time_T}, Table \ref{tab:regular_sampling_time_T}): The time spent on sampling remains negligible compared to the total computational time, even with an increased number of samples.  
    
These results highlight the efficiency of our sampling method, achieving stable and robust performance with little computational cost.

\begin{table}[t]
\centering
\caption{Cut value results corresponding to the times of sample $T$ on Gset.}
\begin{tabular}{ccccccccc}
\toprule
\multirow{2}{*}{$T$} & \multicolumn{2}{c}{\texttt{G70}} & \multicolumn{2}{c}{\texttt{G72}} & \multicolumn{2}{c}{\texttt{G77}} & \multicolumn{2}{c}{\texttt{G81}} \\
 \cmidrule(l){2-3}  \cmidrule(l){4-5}  \cmidrule(l){6-7} \cmidrule(l){8-9}
    & $k=2$ & $k=3$           & $k=2$ & $k=3$           & $k=2$ & $k=3$           & $k=2$ & $k=3$           \\
\midrule
0   & 8912.62& 9968.11& 6099.88&7304.45&8736.58&10323.61&12328.83& 14458.09           \\
1   & 8911  & 9968            & 6100  & 7305            & 8736  & 10321           & 12328 & 14460           \\
5   & 8915  & 9969            & 6102  & 7304            & 8740  & 10326           & 12332 & 14462           \\
10  & 8915  & 9971            & 6102  & 7305            & 8740  & 10324           & 12332 & 14459           \\
25  & 8915  & 9971            & 6102  & 7307            & 8740  & 10326           & 12332 & 14460           \\
50  & 8915  & 9971            & 6102  & 7307            & 8740  & 10327           & 12332 & 14461           \\
100 & 8916  & 9971            & 6102  & 7308            & 8740  & 10327           & 12332 & 14462           \\
\bottomrule
\end{tabular}
\label{tab:gset_objective_T}
\end{table}

\begin{table}[t]
\centering
\caption{Sampling time results corresponding to the times of sample $T$ on Gset.}
\begin{tabular}{ccccccccc}
\toprule
\multirow{2}{*}{$T$} & \multicolumn{2}{c}{\texttt{G70}} & \multicolumn{2}{c}{\texttt{G72}} & \multicolumn{2}{c}{\texttt{G77}} & \multicolumn{2}{c}{\texttt{G81}} \\
 \cmidrule(l){2-3}  \cmidrule(l){4-5}  \cmidrule(l){6-7} \cmidrule(l){8-9}
    & $k=2$ & $k=3$           & $k=2$ & $k=3$           & $k=2$ & $k=3$           & $k=2$ & $k=3$           \\
\midrule
1   & 0.0011 & 0.0006         & 0.0011 & 0.0006         & 0.0020 & 0.0010         & 0.0039 & 0.0020         \\
5   & 0.0030 & 0.0029         & 0.0029 & 0.0030         & 0.0053 & 0.0053         & 0.0099 & 0.0098         \\
10  & 0.0058 & 0.0059         & 0.0058 & 0.0058         & 0.0104 & 0.0104         & 0.0196 & 0.0196         \\
25  & 0.0144 & 0.0145         & 0.0145 & 0.0145         & 0.0259 & 0.0260         & 0.0489 & 0.0489         \\
50  & 0.0289 & 0.0289         & 0.0288 & 0.0289         & 0.0517 & 0.0518         & 0.0975 & 0.0977         \\
100 & 0.0577 & 0.0577         & 0.0576 & 0.0578         & 0.1033 & 0.1037         & 0.1949 & 0.1953         \\
\bottomrule
\end{tabular}
\label{tab:gset_sampling_time_T}
\end{table}

\begin{table}[tb]
\centering
\caption{Cut value results corresponding to the times of sample $T$ on random regular graphs.}
\begin{tabular}{ccccccc}
\toprule
\multirow{2}{*}{$T$} & \multicolumn{2}{c}{$n=100$} & \multicolumn{2}{c}{$n=1,000$} & \multicolumn{2}{c}{$n=10,000$} \\
 \cmidrule(l){2-3}  \cmidrule(l){4-5}  \cmidrule(l){6-7}\
              & $k=2$ & $k=3$               & $k=2$ & $k=3$               & $k=2$ & $k=3$               \\
\midrule
0&126.71&244.77&1291.86&2408.71&12856.53&24102.22 \\ 
1             & 127   & 245                 & 1293  & 2408                & 12856 & 24103               \\
5             & 127   & 245                 & 1293  & 2410                & 12863 & 24103               \\
10            & 127   & 245                 & 1293  & 2410                & 12862 & 24103               \\
25            & 127   & 245                 & 1293  & 2410                & 12864 & 24103               \\
50            & 127   & 245                 & 1293  & 2410                & 12864 & 24103               \\
100           & 127   & 245                 & 1293  & 2410                & 12864 & 24103               \\
\bottomrule
\end{tabular}
\label{tab:regular_objective_T}
\end{table}

\begin{table}[tb]
\centering
\caption{Sampling time results corresponding to the times of sample $T$ on random regular graphs.}
\begin{tabular}{ccccccc}
\toprule
\multirow{2}{*}{$T$} & \multicolumn{2}{c}{$n=100$} & \multicolumn{2}{c}{$n=1,000$} & \multicolumn{2}{c}{$n=10,000$} \\
 \cmidrule(l){2-3}  \cmidrule(l){4-5}  \cmidrule(l){6-7}\
              & $k=2$ & $k=3$               & $k=2$ & $k=3$               & $k=2$ & $k=3$               \\
\midrule
1             & 0.0001 & 0.0001            & 0.0001 & 0.0001            & 0.0006 & 0.0006            \\
5             & 0.0006 & 0.0006            & 0.0007 & 0.0007            & 0.0030 & 0.0030            \\
10            & 0.0011 & 0.0011            & 0.0014 & 0.0013            & 0.0059 & 0.0059            \\
25            & 0.0026 & 0.0026            & 0.0033 & 0.0031            & 0.0145 & 0.0145            \\
50            & 0.0052 & 0.0052            & 0.0065 & 0.0060            & 0.0289 & 0.0289            \\
100           & 0.0103 & 0.0103            & 0.0128 & 0.0122            & 0.0577 & 0.0578            \\
\bottomrule
\end{tabular}
\label{tab:regular_sampling_time_T}
\end{table}

\end{document}